\documentclass[11pt]{amsart}
\usepackage{amsmath,amssymb,amsfonts,amsthm}
\usepackage{a4wide}
\usepackage{fullpage}
\usepackage{pdfsync}
\usepackage{hyperref}
\newtheorem{theorem}{Theorem}
\newtheorem{corollary}{Corollary}
\newtheorem{lemma}{Lemma}
\newtheorem{proposition}{Proposition}
\newtheorem{remark}{Remark}
\makeatletter

\@addtoreset{equation}{section}
\makeatother
\newcommand{\N}{{\mathbb N}}

\newcommand{\R}{{\mathbb R}}

\newcommand{\pa}{{\partial}}
\newcommand{\na}{{\nabla}}

\newcommand{\eps}{{\varepsilon}}
\newcommand{\Hc}{\mathcal{H}}

\renewcommand{\a}{\alpha}

\renewcommand{\l}{\lambda}

\newcommand{\e}{\varepsilon}

\def\curl{\hbox{curl \!}}
\def\div{\hbox{div  }}

\def\divy{\hbox{div}_y }

\def\np{\dot{n}}
\def\up{\dot{u}}
\def\fp{\dot{\phi}}

\def\ZZ{\mathcal{Z}}

\begin{document}

\title{Quasineutral limit of the Euler-Poisson system for ions\\ in a domain with boundaries}
\author{David G\'erard-Varet}
\address{D. G\'erard-Varet \\ IMJ, Universit\'e Denis Diderot, 175 rue du Chevaleret, 75013 Paris} 
\email{gerard-varet@math.jussieu.fr}
\author{Daniel Han-Kwan}
 \address{D. Han-Kwan \\ DMA, ENS, 45 rue d'Ulm, 75005 Paris}
\email{daniel.han-kwan@ens.fr}
 \author{Fr\'ed\'eric Rousset}
  \address{F. Rousset \\ IRMAR, 263 Avenue du G\'en\'eral Leclerc, 35042 Rennes}
  \email{frederic.rousset@univ-rennes1.fr}
  
  \maketitle
  
\begin{abstract}
We study the quasineutral limit
%, that is the limit when the Debye length vanishes, 
of the isothermal Euler-Poisson system describing a plasma made of ions and massless electrons. The analysis is achieved in a domain of $\R^3$ and thus 
extends  former results by  Cordier and Grenier  [Comm. Partial Differential Equations, 25 (2000),
pp.~1099--1113], who dealt with the same problem  in a one-dimensional domain without boundary.
\end{abstract}

\section{Introduction}
We  consider the isothermal Euler-Poisson system which models the behaviour of ions in a background of massless electrons. The massless assumption implies that the electrons follow the classical Maxwell-Boltzmann relation: denoting by $n_e$ their density and after some suitable normalization of the constants, this reads
$$
n_e \: = \:   e^{-\phi},
$$
where $-\phi$ is the electric potential (auto-induced by the charged particles), which satisfies some \emph{Poisson equation}.

Introducing  the characteristic observation length $L$ and the Debye length $\l_D= \sqrt{\frac{\eps_0 k_B T^i}{N_i e^2}}$ of the plasma  (where $T^i$ is the average temperature of the ions, and $N_i$ their average density), we are interested in the study of the behaviour of the system, when 
$$
\frac{\l_D}{L} = \e \ll 1,
$$
which corresponds to the (usual) physical situation where the observation length is much larger than the Debye length (indeed, in usual regimes, $\l_D \sim 10^{-5}-10^{-8} m$). Loosely speaking, the Poisson equation satisfied by $-\phi$ then reads:
$$
\eps^2 \Delta_x \phi \: = \: n_i\: - \: n_e.
$$
In situations where $\eps \ll 1$, this indicates that the plasma can be considered as being ``almost'' neutral (that is $n_i \sim n_e$): hence the name quasineutral limit for the limit $\eps \rightarrow 0$.
However, in a domain with boundaries, it is well-known that there can be some interaction between the plasma and the boundary, so that small scale effects are amplified and boundary layers generally appear. This implies that quasineutrality breaks down near the borders.

In this paper, our goal is to rigorously study this phenomenon of boundary layers by investigating their existence and their stability.
More precisely, the system we work on is the following dimensionless isothermal Euler-Poisson system, for $t > 0$ and  $ x = (x_1,x_2,x_3) = (y,x_3) \in \R^3_+:= \R^2 \times \R_+$:
\begin{equation} \label{EP}
\left\{
\begin{aligned}
 & \pa_t n \: + \:  \div(n u) \:  = \: 0, \\
 & \pa_t u  \: + \:  u \cdot \na u  \: + \: T^i \,  \na \ln(n)= \na \phi, \\
& \eps^2 \Delta \phi \: + \: e^{-\phi}  = n, 
\end{aligned}
\right. 
\end{equation}
where $n$ is the density of ions,  $u = (u_1,u_2,u_3) = (u_y, u_3)$ is their velocity field, and $-\phi$ is the electric potential.  

We complete this system with the  boundary conditions 
\begin{equation} \label{BC}
u_3\vert_{x_3 =0} = 0, \quad \phi\vert_{x_3 = 0} = \phi_b, 
\end{equation}
where $\phi_b$ is some prescribed potential.  We assume that   $\phi_b = \phi_b(y)$ is smooth and compactly supported. Otherwise, it is also possible to consider with minor modifications the case where $\phi_b$ is independent of $y$, which corresponds to the perfect conductor assumption:
$$
\na_y \phi =0.
$$
The  boundary condition on $u$ corresponds to a non-penetration condition. From the physical point of view, this can be interpreted as a kind of confinement condition.
Non homogeneous  boundary conditions are also physically pertinent, and some cases could be treated in our framework: we refer to Remark \ref{rkother} and Section \ref{other}.  Also we shall focus in our analysis on the three-dimensional case  and the simplest   domain  $\mathbb{R}^3_{+}$.
Nevertheless, our analysis can be easily extended to any dimension or more general domains.

\bigskip
%Biblio, du blabla √† √©crire
The Euler-Poisson system and its asymptotic limits have attracted a lot of attention over the past two decades, possibly due to its many applications to plasmas and semiconductors (see for instance \cite{LL,Mar}). Among many possibilities, the so-called quasineutral limit, zero-electron-mass limit, zero-relaxation-time limit have been particularly studied.

For the sake of brevity, we only focus here on the quasineutral limit and describe some mathematical contributions.
For the corresponding Euler-Poisson system describing electrons (with fixed ions), in a domain without boundary,  there have been many mathematical studies of the quasineutral limit (for various pressure laws) and the situation is rather well understood; we refer for instance to \cite{JP,PWaa,Loe,WAN,PW2} and references therein.
The quasineutral limit for Two-fluid Euler-Poisson systems (one for ions and one for electrons), still without boundary, has been also recently investigated in \cite{JJLL1,JJLL2}.

Considering the dynamics of ions with electrons following a Maxwell-Boltzmann law (which is the one we are particularly interested in), Cordier and Grenier studied in \cite{CG} the quasineutral limit of the Euler-Poisson system \eqref{EP}, in the whole space $\R$. A kinetic version of the problem was also recently studied in \cite{HK}.
We also mention the work \cite{GP} of Guo and Pausader who recently constructed global smooth irrotational solutions to \eqref{EP} with $\e=1$ in the whole space $\R^3$. 

\medskip 
For quasineutral limits in presence of boundaries, there have been very few rigorous studies, at least to the best of our knowledge.
When the Poisson equation is considered alone (the density of the ions being prescribed), a boundary layer analysis was led by Br\'ezis, Golse and Sentis \cite{BGS} and Ambroso, M\'ehats and Raviart \cite{AMR}.
For the Poisson equation coupled with some equations for the plasma flow, the are only two main results we are aware of.
The first one is due to Slemrod and Sternberg \cite{SS}, who dealt with the quasineutral limit of some Euler-Poisson models (with massless electrons) in a one-dimensional and stationary case (see also Peng \cite{Peng} for general boundary conditions). The second is due to Peng and Wang \cite{PW} and Violet \cite{VIO}, who studied the quasineutral limit of some Euler-Poisson model (with fixed ions) for stationary and irrotational flows.

To conclude this short review, let us finally mention some other interesting directions  which are related to the issue of the quasineutral limit in a domain.
In \cite{SUZ}, Suzuki studies the asymptotic stability of specific stationary solutions to some Euler-Poisson system, which are interpreted as ``boundary layers''  (see also the paper of Ambroso \cite{AMB} for a numerical study).
In \cite{HS,FHS} and references therein, Slemrod et al. study the problem of formation and dynamics of the so-called plasma sheath.
We refer also to the work  \cite{CDV} of Crispel, Degond and Vignal for a study of a two-fluid quasineutral plasma, based on a formal expansion.

\bigskip
In this paper, we study the quasineutral limit of \eqref{EP}-\eqref{BC}.
As far as we know, our results  which are described in  the next section, give the first description and stability analysis of boundary layers for the quasineutral limit  for the full model  in a non stationary setting (and we are in a multi-dimensional framework as well).

\section{Main results and strategy of the proof}

By putting \emph{formally} $\eps=0$ in the Poisson equation of \eqref{EP}, we first directly obtain the \emph{neutrality} condition:
$$
n=e^{-\phi},
$$
so that one gets from \eqref{EP} in the limit $\eps \rightarrow 0$ the following isothermal Euler system:
\begin{equation} \label{EI}
\left\{
\begin{aligned}
 & \pa_t n \: + \:  \div(n u) \:  = \:  0, \\
 & \pa_t u  \: + \:  u \cdot \na u  \: + \: (T^i+1) \na \ln(n)= 0. 
\end{aligned}
\right. 
\end{equation}
It is natural to expect the convergence to this system. The previous system is well-posed with the only boundary condition
$$ u_{3}\vert_{x_3=0} = 0.$$ 
Because of the boundary conditions \eqref{BC}, we would like to satisfy also the condition
$$ \quad 
\phi\vert_{x_3=0}= \phi_b.
$$
Nevertheless, the solution  $(n, u)$ of \eqref{EI}  cannot  in general satisfy in addition $n\vert_{x_3=0} = e^{-\phi_b}$ and hence,  we  expect the formation of a boundary layer in order to correct this boundary condition.

\bigskip

We shall consider
 solutions of \eqref{EI} with no vacuum: we fix  a reference  smooth  function  $n^{ref}$ which is bounded from
  below by a positive constant and  consider solutions of  \eqref{EI} for which $n$ is bounded from below by a positive constant and
    such that $n- n^{ref}$ is  in $H^s$.  The main result contained in this paper is a rigorous proof of the convergence to the isothermal Euler system \eqref{EI}. 

\begin{theorem}
\label{theomain}
Let $(n^{0}, u^{0})$  a solution to \eqref{EI} such that
$(n^0- n^{ref}, u^0)\in C^0([0,T],H^s(\R^3_{+}))$
 with $s$ large enough. There is $\e_0>0$ such that for any $\e \in (0,\e_0]$, there exists $(n^\e,u^\e)$ a solution to \eqref{EP}-\eqref{BC} also defined on $[0,T]$ such that
 \begin{equation}
  \begin{aligned}
&\sup_{[0,T]}\left( \| n^\e- n^0 \|_{L^2(\R^3_+)}  +   \| u^\e- u^0 \|_{L^2(\R^3_+)}\right) \rightarrow_{\eps \rightarrow 0} 0.
  \end{aligned}
  \end{equation}
  Furthermore, the rate of convergence is $O(\sqrt{\e})$.
    \end{theorem}

\begin{remark}
Let us recall that following the classical theory of the initial boundary value problem for  the compressible Euler equation (see \cite{SCH,Gues}
 for example), asking  $(n^{0}, u^{0})$ to be a smooth solution  to the compressible Euler system  imposes some compatibility conditions on  the initial data $(n^{0}\vert_{t=0}, u^{0}\vert_{t=0})$.
\end{remark}

Boundary layers are hidden in the statement of Theorem \ref{theomain} (recall that they are expected to appear since the  boundary conditions do not match). Nevertheless, in order to give a proof, we will have to describe them precisely.
Indeed, for the system \eqref{EP}  in multi-d,   the only existence result that  is known is the local existence result of
 smooth ($H^m$, $m > 5/2$) solution that can be obtained from the theory of the compressible Euler equation (see \cite{SCH, Gues} for example).
A  nontrivial part of the statement of Theorem \ref{theomain}  is thus that   the solution of   \eqref{EP}
 is defined on an interval of time independent of $\eps$. In  $\mathbb{R}^3$  this  is a consequence of the fact that  \eqref{EP} enjoys  $H^s$ estimates
  that are uniform in $\eps$. These estimates were  obtained in \cite{CG} by  using that \eqref{EP} can be rewritten
   as a symmetric  hyperbolic system perturbed by a large skew-symmetric operator  and thus can be cast in 
    the general framework   of \cite{Grenierpseudo} that generalizes the classical theory of Klainerman and Majda \cite{Klainerman-Majda}
     on the incompressible limit. When boundaries are present
      it seems unlikely to derive directly uniform  $H^s$ estimates of  \eqref{EP}. Indeed the expected
       description of the solution  $(n^\eps, u^\eps, \phi^\eps)$ of \eqref{EP} is under the form 
      \begin{equation} 
      \label{formelintro} n^\eps=  n^0(t,x)+ N^0\left(t,y, {x_{3} \over  \eps}\right)+ \mathcal{O}(\eps) , \quad \phi^\eps= \phi^0(t,x)+ \Phi^0\left(t,y, {x_{3} \over  \eps}\right)+ \mathcal{O}(\eps) , 
        \quad u^\eps =u^0(t,x)+ \mathcal{O}(\eps)
        \end{equation}
        where $(n^0, u^0)$ is the solution of \eqref{EI}, $ \phi^0= - \log n^0$ and the profiles 
         $N^0$, $\Phi^0$ (the  boundary layer profiles) are  smooth and fastly decreasing in the last variable and are added in order to match the
          boundary condition \eqref{BC}. In particular, we require
         $$   \phi^0(t,y,0)+ \Phi^0\left(t,y,  0\right)= \phi^b(y).$$
Note that  at leading order, there is no boundary layer for the velocity but that in general, since the solution of \eqref{EI}  does not satisfy $n\vert_{x_3=0} = e^{-\phi_b}$, the profiles $\Phi^0$
 and $N^0$ are not zero for positive times. This yields that  $n^\eps$ cannot satisfy uniform $H^m$, $m > 5/2$ estimates.
 
 In order to prove Theorem \ref{theomain}, we shall  use a two step argument. We shall first prove that,  up to any order, there exists  
    an approximate solution    $(n^\eps_{app}, u^\eps_{app}, \phi^\eps_{app})$   for \eqref{EP}, 
  in the form of a two-scale asymptotic expansion, 
  and then, we shall prove that  we  can get a true solution of \eqref{EP} defined on $[0, T]$
   by adding a small corrector to the approximate solution. The description of the solution  that we get is thus more precise than
    stated in Theorem \ref{theomain}. We refer  to Theorem \ref{main} and Corollary \ref{coro} which are stated in Section \ref{stability}.     
In particular, as a simple corollary of our estimates, we will  get that  the behavior  in $L^\infty$ is given by:
 \begin{equation}
  \begin{aligned}
\sup_{[0,T_0]}\Big(  \left\| n^\eps- n^0 - N^0\left(t,y,\frac{x_3}{\eps}\right)\right\|_{L^\infty(\R^3_+)}  & \: +  \:  \| u^\eps- u^0 \|_{L^\infty(\R^3_+)}  \\
  & \: + \:   \left\| \phi^\eps- \phi^0 - \Phi^0\left(t,y,\frac{x_3}{\eps}\right)\right\|_{L^\infty(\R^3_+)}\Big)\rightarrow 0.
  \end{aligned}
  \end{equation}
  
  This two step approach has been  very useful   to deal with other  asymptotic  problems involving  as above boundary layers  of size $\eps$ and amplitude
   ${O}(1)$  like   the  vanishing viscosity limit to general hyperbolic systems when the boundary is non-characteristic
   \cite{Gisclon-Serre,Grenier-Gues,Grenier-Rousset,Metivier-Zumbrun}, see also \cite{Temam-Wang},  or the  highly rotating fluids limit
    \cite{Grenier-Masmoudi,Gerard-Varet,Rousset}.  A common feature of these works is  the fact that the  true solution of the system remains close to the approximate one
    on an interval of time that does  not shrink when $\eps$ goes to zero (i.e. the stability of the approximate solution) requires a spectral assumption on the  main boundary layer
     profile  (that is automatically satisfied when the  amplitude of the boundary layer  is sufficiently weak). 
    One quite striking fact about  the quasineutral limit  is that  we shall establish here the   existence and   \emph{unconditional stability} of 
    sufficiently accurate approximate solutions involving  boundary layers of arbitrary amplitude.
\medskip

Since the boundary layer profiles solve ordinary differential equations, we shall be  able to get their existence (without  any restriction on their amplitude)
  by using  classical properties of planar Hamiltonian systems.
 In order to prove  the stability of the  approximate solution, 
  there are two difficulties:  the  fact that the electric field is very large when $\eps$ is small
  (this difficulty is already present when there is no boundary) and the fact that the  linearization of \eqref{EP} about  the approximate solution
   contains zero order terms that involve  the gradient  of the approximate solution and thus that are also  singular when $\eps$ goes
    to zero  because of the presence 
    of boundary layers.
  We shall
 avoid the pseudo-differential approach of \cite{CG} which does not seem to be easily adaptable  when there is a boundary
    and  get an  $L^2$ type estimate which is uniform in $\eps$  by  using  a suitably   \emph{modulated} linearized version of the physical energy of  the Euler-Poisson system. We refer for example  to \cite{Grenier,Chiron-Rousset} for other uses of this approach in singular perturbation problems. 
     In order to estimate higher order derivatives,  we shall  proceed in a classical way by estimating  conormal derivatives of the solution
     and then by using the equation to recover estimates for normal derivatives. Since  the boundary is characteristic  this does not provide
      any information on  the normal derivative of the tangential velocity but we can estimate it by using the vorticity equation.

\medskip

Theorem \ref{theomain} (as well as Theorem \ref{main} and Corollary \ref{coro}) can be generalized to more general settings; this is the purpose of the two following remarks.

\begin{remark}[General domains]
Less stringent settings for the domain could also be treated. Instead of the half-space $\R^3_+$, we can as well consider the case of a smooth domain $\Omega \subset \R^3$ whose boundary is an orientable and compact manifold. We refer to Section \ref{ouvertregulier} for some elements on how one could adapt the proof.
\end{remark}

\begin{remark}[Non-homogeneous boundary condition on $u$]
\label{rkother}
Finally, it is possible to adapt the proof in order to study, instead of the non-penetration condition on $u$, the case of some subsonic  outflow   condition on the normal velocity  $u_{3}$. We refer to Section \ref{other} for some details on what changes in the proof.
\end{remark}

This paper is organized as follows. Section \ref{construction} is dedicated to the construction of a high order boundary layer approximation to a solution of \eqref{EP} (see Theorem \ref{deriv}). In Section \ref{stability}, we prove the stability of this approximation ({\it cf} Theorem \ref{main}), thus proving the convergence to isothermal Euler.  The possible extensions and variations about the result are discussed in the last section.

\section{Construction of boundary layer approximations}
\label{construction}
We look for approximate solutions of system \eqref{EP} of the form
\begin{equation} \label{ansatz}
\begin{aligned}
(n_{a},u_{a}, \phi_{a})    & \: = \:  \sum_{i=0}^K \eps^i \left( n^i(t,x), u^i(t,x), \phi^i(t,x) \right) \\  & \: + \: \sum_{i=0}^K \eps^i \left( N^i\left(t,y,\frac{x_3}{\eps}\right), U^i\left(t,y,\frac{x_3}{\eps}\right), \Phi^i\left(t,y,\frac{x_3}{\eps}\right)\right) 
\end{aligned}
\end{equation}
where $K$ is an arbitrarily  large integer. The coefficients $(n^i, u^i, \phi^i)$ of the first sum depend on $x$. They should model the macroscopic behaviour of the solutions.  The coefficients $(N^i, U^i, \Phi^i)$  of the second sum  depend on $t,y$, but also on a   stretched variable $z = \frac{x_3}{\eps} \in \R_+$. They should model a boundary layer  of size $\eps$ near the boundary. Accordingly, we shall ensure that 
\begin{equation} \label{decay}
 (N^i, U^i, \Phi^i) \rightarrow 0, \quad \mbox{  as }  \: z \rightarrow +\infty 
 \end{equation}
 fast enough.  In order for  the whole approximation to satisfy \eqref{BC}, we shall  further impose
\begin{equation} \label{dirichlet}
\begin{aligned}
& u^i_3(t,y,0) \: + \: U^i_3(t,y,0) \: = \:  0 \quad  \mbox{¬¨‚Ä†for all  } i \ge 0, \\
& \phi^0(t,y,0) \: + \: \Phi^0(t,y,0) \: = \:  \phi_b, \quad   \phi^i(t,y,0) \: + \: \Phi^i(t,y,0) \: = \:  0   \quad \mbox{¬¨‚Ä†for all } i \ge 1. 
\end{aligned}      
\end{equation}
Thus, far away from the boundary,  the term $(n^0, u^0, \phi^0)$  will govern the dynamics of these approximate solutions.  As explained in the previous section, we expect $(n^0, u^0)$ to satisfy the following compressible Euler system
\begin{equation} \label{E}
\left\{
\begin{aligned}
 & \pa_t n \: + \:  \div(n u) \:  = \:  0, \\
 & \pa_t u  \: + \:  u \cdot \na u  \: + \: (T^i+1) \na \ln(n)= 0,  
\end{aligned}
\right. 
\end{equation}
together with the relation $n^0 = e^{-\phi^0}$.  As we will see below, it will also satisfy the non-penetration condition 
\begin{equation} \label{BCE}
u_3\vert_{x_3 = 0} = 0.
\end{equation}

\bigskip

Our results are gathered in the:

\begin{theorem}
\label{deriv}
Let $K \in \N^*, m \in \N^*$ such that $m\geq 3$. For any  data $(n^0_0,u^0_0)$ satisfying  some compatibility conditions on $\{x_3=0\}$
 and such that  $(n^0_{0}- n^{ref}, u^0_{0}) \in H^{m+2K+ 3}(\R^3_+)$, there is $T>0$ and a smooth approximate   solution of \eqref{EP}
  of order $K$
 under the form \eqref{ansatz} such that

{\bf i)} $(n^{0}- n^{ref}, u^{0}) \in C^0([0,T], H^{m+2K+ 3}(\R^3_+))$,  $\phi^0=- \log n^0$,  $(n^0,u^0)$ is a solution to the isothermal Euler system \eqref{E} with initial data $(n^0_0,u^0_0) $.

{\bf ii)}  $\forall \, 1\leq i \leq K$, $(n^i, u^i, \phi^i)  \in  C^0([0,T], H^{m+3}(\R^3_+)).$

{\bf iii)} $\forall \, 0\leq i \leq K, \, N^i,U^i,\Phi^i$ are smooth functions that  belong together with their derivatives  to the set of  uniformly exponentially decreasing functions  with respect to the last variable.
%\begin{align*}
%\Big\{ F(t,y,z) \in L^\infty([0,T], &L^\infty(\R^2\times \R^+)), \\ &\forall t,y \in [0,T]\times\R^2, \, \exists C>0,\mu>0, \forall z \in \R^+, |F(t,y,z)|\leq e^{-\mu \, z} \Big\}.
%\end{align*}

{\bf iv)} Let us consider $(n^\eps, u^\eps, \phi^\eps)$ a solution to \eqref{EP} and define:
\begin{equation*}
n  \: = \:  n^\eps \:  - n_a , \quad u  \: = \: u^\eps  \: -  u_a , \quad \phi  \: = \:  \phi^\eps \: - \phi_a.
\end{equation*}
Then $(n,u,\phi)$ satisfies the system of equations:
\begin{equation} 
\left\{
\begin{aligned}
& \pa_t  n +  (u_a + u ) \cdot \nabla n + n \,\div (u + u_{a})    + \div  (n_{a} u)  = \eps^{K} R_n,  \\
& \pa_t u + (u_a + u)  \cdot \na u + u \cdot \nabla u_{a}  + T^i   \left(\frac{\na n}{n_a +  n } - \frac{\na n_a}{n_a} \left( \frac{n}{n_a +  n} \right) \right) =  \na \phi + \eps^{K}R_u, \\
& \eps^2 \Delta \phi = n - e^{-\phi_a}( e^{ - \phi} - 1 \big) + \eps^{K+1} R_\phi.
\end{aligned}
\right.
\end{equation}
where $R_n,R_u,R_\phi$ are remainders satisfying:
\begin{equation}
\label{estirestes}
\sup_{[0,T]}\| \nabla_x^\a R_{n,u,\phi} \|_{L^2(\R^3_+)} \leq C_a \eps^{-\a_3}, \quad \forall \a=(\a_1,\a_2,\a_3)\in\N^3, \quad |\a|\leq m,
\end{equation}
with $C_a>0$ independent of $\e$.

\end{theorem}

The rest of this section is devoted to the proof of this theorem.

\subsection{Collection of inner and outer equations}
First, we derive formally  the equations that the leading terms of the expansion \eqref{ansatz} should satisfy, to make it an approximate solution of \eqref{EP}-\eqref{BC}. The solvability of these equations will be discussed in the next subsection. 

\medskip
The starting point of the formal derivation is to plug the expansion \eqref{ansatz} into \eqref{EP}. By considering terms that have the same amplitude (that is the same power of $\eps$ in front of it),  one obtains a
collection of equations on the coefficients  $(n^i, u^i, \phi^i)$  and $(N^i, U^i, \Phi^i)$. As usual, we distinguish between two zones: 
\begin{enumerate}
\item An {\em outer zone}  (far from the boundary), in
 which the boundary layer correctors are neglectible. 
 \item An  {\em inner zone} (in the boundary layer),  in which we use the Taylor expansion
$$   n^i(t,y, x_3) \: = \:   n^i(t,y, \eps z) \: = \:  n^i(t,y, 0)  \: + \: \eps  \pa_3   n^i(t,y, 0) z \: + \: \dots \:  \mbox{¬¨‚Ä†(the same for $u^i$ and $\phi^i)$ }$$
 to obtain  boundary layer equations in variables $(t,y,z)$. For clarity of exposure, for any function $f=f(t,x)$, we will denote by    
 $\Gamma f$ the function  $(t,y,Z) \mapsto  f(t,y,0)$.
 
 \end{enumerate}
In the outer zone, by collecting terms that have $O(1)$ amplitude, one sees that $(n^0, u^0)$ satisfies \eqref{EP}, and that $n^0 =  e^{-\phi^0}$.  

\medskip
In the inner zone, by collecting terms that have $O(\eps^{-1})$ amplitude, one has the relations
\begin{equation*}
\left\{ 
\begin{aligned}
& \pa_z \left((\Gamma n^0+ N^0) \, (\Gamma u^0_3 + U^0_3)\right) \: =  \: 0 \\
 & (\Gamma u^0_3 + U^0_3) \,  \pa_z U^0_3 \: + \: T^i \pa_z \ln(\Gamma n^0+ N^0)  =  \pa_z \Phi^0 
\end{aligned}
\right.
\end{equation*}
The first line implies that $(\Gamma n^0+ N^0) (\Gamma u^0_3 + U^0_3)$ does not depend on $z$. From condition \eqref{dirichlet}, we deduce 
\begin{equation} \label{n0u0}
 (\Gamma n^0+ N^0) \, (\Gamma u^0_3 + U^0_3) = 0. 
\end{equation} 
Now, the second line also reads 
\begin{equation} \label{u0n0}
 (\Gamma n^0+ N^0) \, (\Gamma u^0_3 + U^0_3) \, \pa_z U^0_3 \: + \: T^i \pa_z  N^0   =  (\Gamma n^0+ N^0) \,  \pa_z \Phi^0. 
 \end{equation}
Combined with the previous equation, this gives
$$ T^i \pa_z  N^0   =  (\Gamma n^0+ N^0) \pa_z \Phi^0 $$
or equivalently (still using \eqref{decay})
\begin{equation} \label{N0}
 (N^0 + \Gamma n^0) =  \Gamma n^0 e^{\Phi^0/T^i}  
 \end{equation}
 Back to \eqref{u0n0}, we obtain that $\Gamma u^0_3 + U^0_3 = 0$, and by \eqref{decay}, both $U^0_3 = 0$ and $\Gamma u^0_3 = 0$. We recover as expected  that $(n^0, u^0)$ satisfies formally \eqref{E}, together with the non-penetration condition \eqref{BCE}, and  the relation $n^0 = e^{-\phi^0}$.  
 
\medskip 
Still in the inner zone, collecting the $O(1)$ terms in the third line of \eqref{EP}, we obtain 
$$ \pa^2_z \Phi^0 + e^{-\Gamma \phi^0} e^{-\Phi^0} \: = \: N^0 +  \Gamma n^0.  $$
 As $n^0 = e^{-\phi^0}$, this last equation can be simplified into 
 $$ \pa^2_z \Phi^0  + \Gamma n^0 e^{-\Phi^0}  = N^0 + \Gamma n^0$$
Together with \eqref{N0}, we end up with 
\begin{equation} \label{Phi0}
\pa^2_z \Phi^0 \: + \:  \Gamma n^0 \mathcal{S}(\Phi^0)  = 0, \quad  \mathcal{S}(\Phi) \: := \: e^{-\Phi} - e^{\Phi/T^i}. 
\end{equation}
 Note that this is a second order equation in $z$, so that with the boundary conditions 
 \begin{equation} \label{BCPhi0}
 \Phi^0\vert_{z=0} = \phi_b - \phi^0\vert_{x_3=0}, \quad  \Phi^0\vert_{z=+\infty} = 0  
 \end{equation}
 it should determine completely $\Phi^0$. Indeed,  the existence of a unique smooth solution of \eqref{Phi0}-\eqref{BCPhi0} with rapid decay at infinity will be established in the next subsection.  As $\Phi^0$ is determined, so is $N^0$ by relation \eqref{N0}.

 \medskip
 This concludes the formal derivation of  $(n^0,u^0,\phi^0)$, and $(U^0_3, \Phi^0, N^0)$.   
 We will now turn to the equations satisfied by  $(n^1, u^1, \phi^1)$ and $(U^1_3, U^0_y, \Phi^1, N^1)$. More generally,  we will  derive for all $i \ge 1$:
 \begin{enumerate}
\item   a set of equations on $(n^i, u^i, \phi^i)$, whose source terms depend on the coefficients  
 $(n^k,u^k,\phi^k)$ and their derivatives, $\: k \le i-1$. 
 \item  a set of equations on  $(U^i_3, U^{i-1}_y, \Phi^i, N^i)$, 
 whose sources depend on (the trace at $x_3= 0$ of) $(n^k,u^k,\phi^k)$ and their derivatives,  $\: k \le i$, as well as on  $(U^k_3, U^{k-1}_y, \Phi^k, N^k)$ and their derivatives, $\: k \le i-1.$ 
 \end{enumerate}
 Indeed, in the outer  zone, collecting terms with  amplitude  $O(\eps^i)$ in \eqref{EP} leads to systems of the type:
 \begin{equation} \label{outeri}
 \left\{
 \begin{aligned}
& \pa_t n^i + \div(n^0 u^i) + \div(n^i u^0) \: = \: f_n^i \\ 
&  \pa_t u^i + u^0 \cdot \na u^i  + u^i \cdot \na u^0  + T^i \na  \left( \frac{n^i}{n^0} \right)  = \na \phi^i   + f_u^i  \\
& - e^{-\phi^0} \, \phi^i = n^i  + f_{\phi}^i,
 \end{aligned}
 \right.
 \end{equation}
 where the source terms $f^i_{n,u,\phi}$ depend on  $(n^k,u^k,\phi^k)$ and their derivatives,  $\: k \le i-1$. 
  
 \medskip
 In the inner zone, collecting $O(\eps^i)$ terms in the Poisson equation of \eqref{EP}, we get:
 \begin{equation} \label{Phii} 
 \pa_z^2 \Phi^i  \: -  \:  \Gamma n^0 \, e^{-\Phi^0}  \, \left(  \Phi^i \: + \: \Gamma \phi^i \right) \: = \: \left(N^i + \Gamma n^i \right) \: + \: F_\Phi^i 
 \end{equation}
 where $F_{\Phi}^i$ depend  on    $(n^k,\phi^k)$ and on  $(N^k, \Phi^k)$, $\: k \le i-1$.  Then, collecting the $O(\eps^{i-1})$ in the equation for ion density $n$, taking into account that $U^0_3 = 0$ and  $\Gamma u^0_3 = 0$, we get 
\begin{equation} \label{Ui3}
 \pa_z \left((N^0 + \Gamma n^0)( U^i_3 + \Gamma u^i_3)\right) \: + \: \divy \left((N^0 + \Gamma n^0) (U^{i-1}_y + \Gamma u^{i-1}_y)  \right) \: = \: F^i_{N}.
 \end{equation}
 The equation on $u_y$ yields 
 \begin{equation} \label{U0y}
\begin{aligned} 
\hspace{-1cm} \pa_t (U^{0}_y + \Gamma u^0_y)  \:  & + \:  (U^1_3 + \Gamma u^1_3 + \Gamma \pa_3 u^0_3  ) \,  \pa_z U^0_y   \: \\
 & +  \: (U^0_y + \Gamma u^0_y) \cdot \na_y \left( U^0_y + \Gamma u^0_y \right)  \: + \:  (T^i+1) \na_y \ln(\Gamma n^0) \: = \: 0
 \end{aligned}
 \end{equation} 
 and for $i \ge 2$:
 \begin{equation} \label{Ui-1y}
\begin{aligned}
 \pa_t (U^{i-1}_y + \Gamma u^{i-1}_y)   \: & + \:  (U^i_3 + \Gamma u^i_3 + \Gamma \pa_3 u^{i-1}_3  ) \,  \pa_z U^0_y  \: +  \: (U^0_y + \Gamma u^0_y) \cdot \na_y \left( U^{i-1}_y + \Gamma u^{i-1}_y \right) \\ 
&  +  \: \left( U^{i-1}_y + \Gamma u^{i-1}_y \right)  \cdot \na_y  (U^0_y + \Gamma u^0_y) \: = \:  F^i_y.
 \end{aligned}
 \end{equation}
 Finally, the equation on $u_3$ yields 
\begin{equation} \label{Ni}
  T^i \pa_z \left( \frac{N^i + \Gamma n^i}{N^0 + \Gamma n ^0} \right) \: = \:  \pa_z \Phi^i \: + \:  F^i_3. 
\end{equation}
Again, the source terms   $\: F^i_N$, $F^i_y$ and $F^i_3$ depend on  $\: (n^k,u^k,\phi^k)$,  and on  $\: (U^k_3, U^{k-1}_y, \Phi^k, N^k)$,  for indices $k \le i-1$. Note that  \eqref{decay} requires the following compatibility conditions: 
% In view of  the collection of systems \eqref{outeri},  they will automatically  satisfy the compatibility conditions imposed by \eqref{decay}, that are: 
\begin{equation} \label{compatible sources}
\left\{
\begin{aligned}
& -n^0\phi^i\vert_{x_3=0}  \: =  \: n^i\vert_{x_3=0} +  F^i_\Phi\vert_{z=\infty}, \\
&  \divy\left( n^0 \,  u^{i-1}_y \right)\vert_{x_3=0} \: = \:  F^i_N\vert_{z=\infty}, 
\\
& \left( \pa_t  u^{i-1}_y +  u^0_y \cdot \na_y  u^{i-1}_y +   u^{i-1}_y  \cdot \na_y u^0_y\right)\vert_{x_3=0} \: = \:  F^i_y\vert_{z=\infty}, \\
&  F^i_3\vert_{z=\infty} = 0. 
\end{aligned}
\right.
\end{equation}

\medskip
It remains to show how to deduce the profiles from the equations. First of all, we can derive $U^0_y$ from \eqref{Ui3} (with $i=1$) and \eqref{U0y}. We integrate \eqref{Ui3} from $0$ to $z$, and with 
 \eqref{dirichlet}  deduce 
\begin{equation} \label{Ui3bis}
(U^i_3 + \Gamma u^i_3) \: = \: -\frac{1}{N^0 + \Gamma n^0} \int_0^z \left(  \divy \left((N^0 + \Gamma n^0) (U^{i-1}_y + \Gamma u^{i-1}_y) \right) \: - \:  F^{i}_N \right)  
\end{equation}
For $i=1$, this last relation allows to express $U^1_3 + \Gamma u^1_3$ in terms of $U^{0}_y + \Gamma u^{0}_y$ (and coefficients of lower order). One can then substitute into \eqref{U0y} to have a closed equation on $U^0_y$. We notice that it has the trivial solution  $U^0_y = 0$. {\em In all what follows, we shall restrict to approximate solutions \eqref{ansatz} satisfying $U^0_y = 0$.} 

\medskip
As $U^0_y = 0$, the equation \eqref{Ui-1y} simplifies into
 \begin{equation} \label{Ui-1ybis}
 \pa_t U^{i-1}_y   \:  +  \:  \Gamma u^0_y \cdot \na_y U^{i-1}_y  
  +  \:  U^{i-1}_y  \cdot \na_y  \Gamma u^0_y \: = \:  H^i 
   \end{equation}
with 
$$H^i \ := \ F^i_y -  \left( \pa_t \Gamma u^{i-1}_y   \: + \:  \Gamma u^0_y \cdot \na_y \Gamma u^{i-1}_y \  +  \: \Gamma u^{i-1}_y  \cdot \na_y  \Gamma u^0_y \right).$$

\medskip
One can now derive recursively  $(n^i, u^i, \phi^i)$  and $(U^i_3, U^{i-1}_y, \Phi^i, N^i)$,  $i \ge 1$. 
We take $i \ge 1$, and  assume that all lower order profiles $(n^k,u^k,\phi^k)$ and $(U^k_3, U^{k-1}_y, \Phi^k, N^k)$, $k \le i-1$ are known. 
From  the equations satisfied by the $(n^k, u^k, \phi^k)$'s, the last three compatibility conditions in \eqref{compatible sources} are satisfied. In particular, the function $H^i$ in \eqref{Ui-1ybis}  decays to zero at infinity.  Thus, the linear transport equation \eqref{Ui-1ybis}  yields a  decaying solution $U^{i-1}_y$ (with the special case $U^0_y = 0$ when $i=1$). Then,  one obtains from \eqref{Ui3bis} the expression of  $U^i_3 + \Gamma u^i_3$. From the decay condition \eqref{decay}, we deduce that  
\begin{equation}  \label{BLu3i}
u^i_3\vert_{x_3 = 0} = \frac{-1}{n^0\vert_{x_3 = 0}} \int_0^{+\infty} \left( \divy \left((N^0 + \Gamma n^0) (U^{i-1}_y + \Gamma u^{i-1}_y)\right)  \: - \:  F^{i}_N \right) 
\end{equation}
This  boundary condition goes with  the hyperbolic system on $(n^i, u^i)$ (see \eqref{outeri}). Together with an  initial condition which is compatible with the boundary condition,  it allows to determine $(n^i, u^i, \phi^i)$. This will be explained rigorously in the next subsection.  
 Note that, as $(n^i, u^i, \phi^i )$ solves \eqref{outeri}, the remaining compatibility condition in \eqref{compatible sources} is satisfied.  Using again \eqref{Ui3bis}, one eventually gets $U^i_3$.

\medskip
To end up this formal derivation, it remains to handle $\Phi^i$, and $N^i$.  
  Combining \eqref{Phii} and \eqref{Ni} (together with the decay condition \eqref{decay}) leads to a second order equation on $\Phi^i$:
\begin{equation}  \label{Phiibis}
\pa^2_z \Phi^i \: + \: \Gamma n^0 \mathcal{S}'(\Phi^0)  \: = \:  G_{\Phi} 
\end{equation}
where we remind that $\mathcal{S}(\Phi) = e^{-\Phi} - e^{\Phi/T^i}$ and $G_\Phi$ is a source term that depends on the lower order terms, decaying to zero as $z$ goes to infinity. Again, the solvability of this system with Dirichlet conditions
\begin{equation} \label{BCPhiibis}
 \Phi^i\vert_{z=0} = -\phi^i\vert_{x_3=0}, \quad \Phi^i\vert_{z=+\infty} = 0 
 \end{equation}
 will be established later on. Once $\Phi^i$ is known, $N^i$ is determined through \eqref{Ni}.

\subsection{Well-posedness of the reduced models}
To complete the construction of the approximate solutions and get Theorem \ref{deriv}, we must establish the well-posedness of the inner and outer systems derived in the previous paragraph. 

\medskip
The well-posedness of the inner systems is classical. 
Indeed, the system \eqref{E}-\eqref{BCE} on $(n^0, u^0)$ is a standard compressible Euler system. For well-posedness, we consider an  initial data of the form  $(n^0_0 = e^{-\phi^0_0}, u^0_0)$, with $(u^0_0, n^0_0- n^{ref}) \in H^{m+3+2K}(\R^3_+)^4$. %(with $s$ large enough). 
Following \cite{Gues,SCH}, if this initial data  satisfies  standard compatibility conditions at the boundary $\{x_3 = 0\}$, there exists a unique  solution $(n^0 = e^{-\phi^0}, u^0)$, with  
$$(n^0- n^{ref},u^0) \in C^\infty\left([0,T]; H^{m+3+2K}(\R^3_+)^4\right), \quad \mbox{¬†for some } T > 0.$$ 
As regards the  next systems $\eqref{outeri}$-\eqref{BLu3i}, they resume to linear hyperbolic systems on $(n^i, u^i)$, $i\ge 1$, with a source made of some derivatives of previously constructed $(n^j, u^j, \phi^j)$, $j< i$. For initial data 
$(n^i_0, u^i_0)  \in H^{m+3+2K- 2i}(\R^3_+)^4$, chosen in order to match compatibility conditions  they  have again unique solutions 
$$ (n^i,u^i) \in C^\infty\left([0,T]; H^{m+3+2K-2i}(\R^3_+)^4\right). $$ 
The ``loss'' of regularity is due to the derivatives in the source.  Likewise, one shows that:
$$ U^i \in C^\infty\left([0,T]; H^{m+3+2K-2i}(\R^3_+)^4\right),$$ 

\medskip
The main point is the resolution of the nonlinear boundary layer system \eqref{Phi0}-\eqref{BCPhi0}. Note that $t$ and $y$ are only  parameters in such a system, involved through the  coefficient $\gamma n^0$ and the boundary data. For each $(t,y)$, \eqref{Phi0} is an ordinary differential equation in $z$, of the type 
\begin{equation} \label{ode}
 \Phi'' + \gamma \mathcal{S}(\Phi) = 0, \quad \gamma > 0,
 \end{equation}
and  we want to prove that, for any constant $\phi$,  it has   a (unique) solution $\Phi$  that connects $\phi$  to $0$, with  exponential decay of $\Phi$ and its derivatives as $z$ goes to infinity. 

 Therefore, we rewrite \eqref{ode}  as a Hamiltonian system, 
\begin{equation} \label{hamilton}
\frac{d}{dz} \begin{pmatrix} p \\ \Phi \end{pmatrix} = \na^\bot H(p,\Phi), \quad p := \: \Phi', \quad H(p,\Phi) \: := \: \frac{p^2}{2} \: + \:  \mathcal{T}(\Phi) 
\end{equation}
where $\mathcal{T}$ is an antiderivative for $\gamma \mathcal{S}$. We choose the one that vanishes at $0$: 
$$   \mathcal{T}(\Phi) \: := -\gamma \left( e^{-\Phi} + T^i e^{\Phi/T^i}\right) + \gamma (1+T^i).  $$
The Hamiltonian $H$ is of course constant along any trajectory.  
The linearization of \eqref{hamilton}  at the critical point $(p,\Phi)=(0,0)$ yields 
$$  \frac{d}{dz} \begin{pmatrix} \dot{p} \\ \dot{\Phi} \end{pmatrix} \: = \begin{pmatrix} 0 & \gamma (1+1/T^i) \\ 1 & 0 \end{pmatrix} \begin{pmatrix} \dot{p} \\ \dot{\Phi} \end{pmatrix}. $$
Hence, we get that  $(0,0)$ is a saddle fixed point, and using the stable manifold theorem that its  stable manifold is locally a curve,  tangent to $\left(\begin{smallmatrix} \sqrt{\gamma (1+1/T^i) } \\ -1 \end{smallmatrix} \right)$. Moreover, as $H(0,0) = 0$,  the stable manifold is exactly  the branch of 
$ \{¬†H(p,\Phi) = 0 \} $ given by the equations
$$  p =  -   \sqrt{\gamma} \sqrt{e^{-\Phi} + T^i e^{\Phi/T^i} - 1 - T^i} \: \mbox{ for  } \: \Phi \ge 0, \quad p =     \sqrt{\gamma} \sqrt{e^{-\Phi} + T^i e^{\Phi/T^i} - 1 - T^i} \: \mbox{ for  } \: \Phi < 0. $$
 In addition, any solution starting on this branch  decays exponentially to $0$, and by using the equation \eqref{ode}, so do all its derivatives.  
Finally, one needs to check  that for any $\psi$,  this branch has a unique point with  ordinate $\psi$. It is indeed the case,  as this branch gives $p$ as a decreasing function of $\Phi$, going to infinity at infinity. This concludes the resolution of the o.d.e, and yields in turn the solution $\Phi^0$ of our boundary layer system. Its  regularity  with respect to $(t,y)$ follows from regular dependence of the solution of \eqref{ode} with respect to the data.

\medskip
Last step is the well-posedness of the boundary layer systems \eqref{Phiibis}-\eqref{BCPhiibis} which allow to define $\Phi^i$ and $N^i$.  Again, $t,y$ are parameters, and the point is to find a (unique) solution to  the inhomogeneous linear ODE
\begin{equation} \label{ode2}
 \Phi'' + \gamma S'(\overline{\Phi}) \Phi = F, 
 \end{equation}
with exponential decay at infinity  and  a Dirichlet condition $\Phi\vert_{z=0} = \psi$. Here,  $\overline{\Phi}$ and $F$ are arbitrary smooth functions of $z$, decaying exponentially to $0$. 

Note that, up to consider $\tilde \Phi(z) = \Phi(z) - \psi \chi(z) $, with $\chi \in C^\infty_c(\R_+)$ satisfying $\chi(0)=0$, we can always assume that $\psi = 0$.  After this simplification, observing that $S'(\overline{\Phi})\leq 0$, we can apply the Lax-Milgram Lemma: it provides a unique solution $\Phi \in H^1_0(\R_+)$ of \eqref{ode2}. By Sobolev embedding, it decays to zero at infinity, and it is smooth by elliptic regularity.  It remains to obtain the exponential decay of $\Psi$, hence we write \eqref{ode2} as
$$ \Phi'' +    \gamma S'(0) \Phi = G, \quad G :=      (\gamma S'(0) -  \gamma S'(\overline{\Phi})) \Phi + F $$
Using the Duhamel formula for this constant coefficient ode, and the exponential decay of $G$, one obtains easily that any bounded solution decays  exponentially. 

Finally one deduces recursively that $U^i$ also decays exponentially.

\medskip

Summing up what has been done in this section, we have proved Theorem \ref{deriv}.

\section{Stability estimates}
\label{stability}
This section is devoted to the stability of the boundary layer approximations built in the previous section. 

Let us write the solution $(n^\eps, u^\eps, \phi^\eps)$ of \eqref{EP} under the form 
\begin{equation*}
n^\eps  \: = \:  n_a \: +  n , \quad u^\eps  \: = \: u_a  \: +  u , \quad \phi^\eps  \: = \:  \phi_a \: + \phi
\end{equation*}
where $n_a, u_a, \phi_a$ are shorthands for $n^\eps_{app}, u^\eps_{app}, \phi^\eps_{app}$ (defined in \eqref{ansatz}). 
Then, we get for $(n, u, \phi)$ the system
\begin{equation} \label{EPP}
\left\{
\begin{aligned}
& \pa_t  n +  (u_a + u ) \cdot \nabla n + n \,\div (u + u_{a})    + \div  (n_{a} u)  = \eps^{K} R_n,  \\
& \pa_t u + (u_a + u)  \cdot \na u + u \cdot \nabla u_{a}  + T^i   \left(\frac{\na n}{n_a +  n } - \frac{\na n_a}{n_a} \left( \frac{n}{n_a +  n} \right) \right) =  \na \phi + \eps^{K}R_u, \\
& \eps^2 \Delta \phi = n - e^{-\phi_a}( e^{ - \phi} - 1 \big) + \eps^{K+1} R_\phi.
\end{aligned}
\right.
\end{equation}
together with the boundary conditions
\begin{equation} \label{EPPbc1}
u_3\vert_{x_3 = 0} = 0, \quad \phi\vert_{x_3=0} = 0
\end{equation}
and the initial condition
\begin{equation} \label{EPPbc2}
u\vert_{t = 0} = \eps^{K+ 1} u_{0}, \quad n\vert_{t=0} = \eps^{K+ 1} n_{0}.
\end{equation}
Observe here that in \eqref{EPP}, $\eps^{K} R_n,  \eps^{K}R_u$ and $\eps^{K+1} R_\phi$ are remainders that appear because $(n_a,u_a,\phi_a)$ is not an exact solution of \eqref{EP}.

\bigskip

The main result of this section is 
\begin{theorem}
\label{main}
Let $m \geq 3$, and $(n_{0}, u_{0})\in H^m(\R^3_{+})$ some initial data for \eqref{EPPbc2}, satisfying some suitable compatibility  conditions. 
Let $K \in \N^*, K \geq m$ and $(n_{a}, u_{a}, \phi_{a})$
 an  approximate solution at order $K$ given by  Theorem \ref{deriv} which is defined on $[0, T_{0}]$. There exists
 $\eps_0$ such that for every $\eps \in(0, \eps_{0}]$, the solution of  \eqref{EPP}-\eqref{EPPbc1}-\eqref{EPPbc2} is defined  on $[0,T_{0}]$
  and satisfies the estimate
  $$ \eps^{|\alpha|} \| \partial^\alpha(n,u, \phi, \eps\nabla  \phi) \|_{L^2(\mathbb{R}^3_{+})} \leq C \eps^{K}, \quad \forall t \in [0, T^0], \quad \forall \alpha
   \in \mathbb{N}^3, \, |\alpha| \leq m.$$ 
  
\end{theorem}
One can then remark that as a simple rephrase of  Theorem \ref{main}, we obtain:

\begin{corollary}
\label{coro}
Let $m\geq 3$. Let $K \in \N^*, K \geq m$ and $(n_{a}, u_{a}, \phi_{a})$
 an approximate solution at order $K$ given by  Theorem \ref{deriv} which is defined on $[0, T_{0}]$. There exists
 $\eps_0$ such that for every $\eps \in(0, \eps_{0}]$, there is a solution $(n^\e,u^\e,\phi^\e)$ to  \eqref{EP} which  is defined  on $[0,T_{0}]$
  and satisfies the estimate:
  \begin{equation}
   \left\| \Big(n^\eps-n_{a},u^\eps-u_{a}, \phi^\eps-\phi_{a} \Big)\right\|_{H^{m}(\mathbb{R}^3_{+})} \leq C \eps^{K-m}, \quad \forall t \in [0, T^0].
   \end{equation}
  In particular, we get the $L^2$ and $L^\infty$ convergences as $\eps \rightarrow 0$:
 \begin{equation}
  \begin{aligned}
&\sup_{[0,T_0]}\left(  \left\| n^\eps- n^0 - N^0\left(\cdot,\cdot,\frac{\cdot}{\eps}\right)\right\|_{L^\infty(\R^3_+)}  +   \| n^\eps- n^0 \|_{L^2(\R^3_+)}\right) \rightarrow 0, \\
&\sup_{[0,T_0]}  \left(    \| u^\eps- u^0 \|_{L^\infty(\R^3_+)}  +   \| u^\eps- u^0 \|_{L^2(\R^3_+)}\right) \rightarrow 0, \\
& \sup_{[0,T_0]}  \left(    \left\| \phi^\eps- \phi^0 - \Phi^0\left(\cdot,\cdot,\frac{\cdot}{\eps}\right) \right\|_{L^\infty(\R^3_+)}   +   \| \phi^\eps- \phi^0 \|_{L^2(\R^3_+)}\right) \rightarrow 0.
  \end{aligned}
  \end{equation}
\end{corollary}
Of course, this contains Theorem \ref{theomain}.

\bigskip

From now on, our goal is to prove Theorem \ref{main}.

For $\eps >0$ fixed, since  in the equation \eqref{EPP}, the  term involving $\phi$ in the second equation can be considered
as a semi-linear term,  the  known local existence results for the compressible Euler equation (\cite{SCH,Gues} for example) can be applied to 
 the system \eqref{EPP}. 
  Let us assume that  $(n_{0}, u_{0}) \in H^m(\R^3_{+})$ for $m \geq 3$ and that it satisfies suitable compatibility conditions on the boundary, 
   then there exists $T^\eps >0$ and  a unique solution of \eqref{EPP} defined on $[0,T^\eps]$ such that
     $ u \in \mathcal{C}([0, T^\eps), H^m)$ and that there exists $M>0$ satisfying %Modifié 3eme critere- il faudrait peut être faire apparaitre la 3eme condition plus lisiblement, et comprendre s'il s'agit bien du critère de Bohm
   \begin{equation}
   \label{apriori2} n_{a}(t,x) + n(t,x) \geq  1/M , \quad  e^{- \phi_{a}}(1+\min(h_0,h_1)) \geq 1/M, \quad    \| \chi( u_{a}+ u)_{3}  \|_{L^\infty} \leq \sqrt{3/4 \, T^i}, 
 \quad \forall t \in [0,T^\eps)
\end{equation}
where 
%% problème signe h_0?
\begin{equation} \label{h0h1}
h_0(\phi) \: := \:  -{ e^{-\phi} - 1 + \phi \over \phi}, \quad h_1(\phi) \: := \:  e^{-\phi}- 1  
\end{equation}
and $\chi(x_3):= \tilde\chi(x_3/\delta)$ where $\tilde\chi$ is  a  smooth compactly supported function, equal to $1$ in the vicinity of  zero and  $\delta>0$ is chosen so that
$$ \| \tilde\chi(x_{3}/\delta)\big(( u_{a}+ u)_{3}\big)\vert_{t=0} \|_{L^\infty} \leq \sqrt{T^i/2}.$$ Note that  we can always chose $\delta$ in this 
way since  at $t=0$ we have $ \big(( u_{a}+ u)_{3}\big)\vert_{t=0, x_{3}=0}=0$.

   The difficulty is thus to prove that  the solution actually exists on an interval of time independent of $\eps$.
    We shall  get this result  by proving  uniform energy estimates combined with the previous local existence result when  the initial data
     and the source term are sufficiently small (i.e. when the approximate solution $ (n_{a}, u_{a})$ is sufficiently accurate).
     Note that there are two difficulties in order to get useful energy estimates. The first one is the singular perturbation coming from
      the Poisson  part, the electric field cannot be considered as a lower order term uniformly in $\eps$. The second
       one comes from the boundary layer terms contained in the approximate solution which create singular terms, for example, 
        in the second line of \eqref{EPP},  the zero order term $ \frac{\na n_a}{n_a} \, \frac{n}{n_a +  n}$
         is singular in the sense that when \eqref{apriori2} is matched,  we only have  the  estimate
         $$  \left\|\frac{\na n_a}{n_a} \, \frac{n}{n_a +  n} \right\|_{L^2} \lesssim {1 \over \eps} \|n\|_{L^2}.$$
    The main step will be the obtention of uniform estimates for  ``linearized'' systems. 
\subsection{Energy for the quasineutral Euler-Poisson system without source}
We start by recalling that the isothermal Euler-Poisson system without source (which corresponds here to the case $\phi_b=0$) has a conserved physical energy, which is given in the
\begin{proposition}
Let $\eps> 0$ and $(n, u,\phi)$ a strong solution to \eqref{EP} on $[0,T]$ with  $\phi_b=0$. We define the energy functional: 
\begin{equation}\label{energy}
\begin{aligned}
\mathcal{E_\eps}(t) \: & := \: \frac{1}{2} \int_{\R^3_+} n |u|^2 dx + T^i \int_{\R^3_+} n(\log n  - 1) dx  + \int_{\R^3_+} (1-\phi)e^{-\phi} dx  + \frac{\eps^2}{2} \int_{\R^3_+} \vert \nabla_x \phi \vert^2 dx.
\end{aligned}
\end{equation}
Then for any $t \in [0,T], \,
\mathcal{E_\eps}(t) = \mathcal{E_\eps}(0)$.
\end{proposition}
This property will never be used in the sequel; nevertheless, the $L^2$ and higher order stability estimates which follow are obtained via the study a \emph{modulated} version of this energy. Thus, for clarity of exposure, we briefly present the proof showing that the energy is conserved, as it is much easier to follow but share the same spirit with the subsequent ones.

\begin{proof}
%This follows from an explicit computation, which seems to be rather classical for the equations posed in a domain without boundary. 
We compute the derivative in time of the first term of $\mathcal{E}$, by using the transport equation satisfied by $n$ in \eqref{EP} (which corresponds to the convervation of charge):
\begin{align*}
\frac{d}{dt}  \int_{\R^3_+} n |u|^2 dx & \: = \: \int_{\R^3_+}\pa_t n |u|^2 dx + 2 \int_{\R^3_+} n u \cdot \pa_t u dx \\
& \: = \: -  \int_{\R^3_+} \nabla \cdot (n u) |u|^2 dx + 2 \int_{\R^3_+} n u \cdot \pa_t u dx.
\end{align*}
Using the non-penetration boundary condition for $u$, we have by integration by parts:
$$
 -  \int_{\R^3_+} \nabla\cdot (n u) |u|^2 dx =    \int_{\R^3_+}  (n u) \cdot \nabla |u|^2 dx =  \int_{\R^3_+}  (n u) u \cdot \nabla u dx.
 $$
Now thanks to the equation satisfied by $u$ in \eqref{EP}, we can write: 
$$
 \int_{\R^3_+} n u\cdot \pa_t u dx = - \int_{\R^3_+}  (n u) u \cdot \nabla u dx - T^i \int_{\R^3_+} n u \cdot \nabla \log n + \int_{\R^3_+} n u \cdot \nabla \phi dx.
$$
By the equation satisfied by $n$ and the non-penetration condition, we infer, by another integration by parts, that
$$
T^i \int_{\R^3_+} n u \cdot \nabla \log n = -T^i \int_{\R^3_+} \pa_t n \log n = -\frac{d}{dt} T^i \int_{\R^3_+} n(\log n  - 1) dx.
$$
Likewise, we get, using once more the conservation of charge,
\begin{align*}
\int_{\R^3_+} n u \cdot \nabla \phi dx & \: = \:  - \int_{\R^3_+} \na \cdot (n u) \phi dx  \\
& \: = \: \int_{\R^3_+} \pa_t n \, \phi dx.
\end{align*}
Deriving with respect to time the Poisson equation in \eqref{EP} we obtain
\begin{align*}
\int_{\R^3_+} n u \cdot \nabla \phi dx & \: = \: \int_{\R^3_+} \eps^2 \pa_t \Delta \phi \phi dx +  \int_{\R^3_+}  \pa_t e^{-\phi} \phi dx.
\end{align*}
The last term of the r.h.s. is treated exactly like the pressure term. Considering the first one, by integration by parts, we obtain:
$$
\int_{\R^3_+}  \pa_t \Delta \phi \, \phi dx = -\frac{1}{2} \frac{d}{dt} \int_{\R^3_+} \vert \nabla_x \phi \vert^2 dx,
$$
which relies on the fact that $\phi=0$ on $\{x_3=0\}$. This proves our claim.
\end{proof}
\subsection{$L^2$ estimate for  the suitably linearized equations}
We establish here an $L^2$ estimate for the solution $(\np,\up,\fp)$ of the following  linearized system: 
\begin{equation} \label{linear}
\left\{
\begin{aligned}
& \pa_t  \np +  (u_{a} + u )\nabla \np + (n_a + n) \nabla \cdot \up+ \up \cdot \nabla (n_{a}+   n)+ \np \,\div (u_{a}+u)   = r_n,  \\
& \pa_t \up + (u_a + u)  \cdot \na \up  + \up \cdot \nabla u_{a} +  T^i   \left(\frac{\na \np}{n_a + n} - \frac{\na n_a}{n_a} \left( \frac{\np}{n_a + n} \right) \right) =  \na \fp + r_u, \\
& \eps^2 \Delta \fp = \np + e^{-\phi_a} \fp (  1 + h(\phi) )  + r_\phi.
\end{aligned}
\right.
\end{equation}
 where $r = (r_n, r_u,r_\phi)$ is a given source term, and where $h \in \{h_0, h_1\}$, {\it cf} \eqref{apriori2}-\eqref{h0h1}. The reason why we shall consider these two possibilities for $h$ will become clear in view of paragraph \ref{MEE}.
We add to the system  the boundary conditions
\begin{equation}
\up_3\vert_{x_3 = 0} = 0, \quad \fp\vert_{x_3=0} = 0.
\end{equation}
The crucial estimate is given by 
\begin{proposition}
\label{propL2}
 Let $(n_{a}, u_{a},\phi_{a})$ the approximate solution constructed in Theorem \ref{deriv}
  and   some smooth $(n, u, \phi)$ such that $u_{3}\vert_{x_{3=0}}=0$ and 
  \begin{equation}
  \label{hypmin}
 n_{a} + n \geq  1/M, \quad e^{-\phi_{a}}(1 + h(\phi)) \geq  1/M,  \quad |n|+  |u|+  |\phi| \leq M ,  \quad \forall t \in [0,T], \, x \in \mathbb{R}^3_{+}.
  \end{equation} 
  Then, there exist $C(M)$ and $ C( C_{a}, M )$  independent of $\eps$  ($C_{a}$ only depends on the approximate solution)  such that we have on $[0,T]$ the estimate
 \begin{align*}
 & \big\| \big( \np,\up, \fp, \eps \nabla \fp) (t) \big\|_{L^2(\R^3_{+})}^2 
   \leq   C(M)\Big(  \big\| \big( \np_{0}, \up_{0}) \big\|_{L^2(\R^3_{+})}^2  + \int_{0}^t \big\| (\eps^{-1}r_{\phi}, \pa_{t} r_{\phi},  r_{n}, r_u\big)\big\|_{L^2(\R^3_{+})}^2 \Big) \\
     & +    C(C_{a}, M)   \int_{0}^t  \big( 1 +   \big\| \nabla_{t,x}\big( n,u,\phi)\big\|_{L^\infty(\R^3_{+})}  +\eps^{-1} \| n\|_{L^\infty(\R^3_{+})}\big)  \big\| \big( \np,\up, \fp, \eps \nabla \fp\big) \big\|_{L^2(\R^3_{+})}^2 \Big)
  \end{align*}

\end{proposition}
Note that this is indeed a stability estimate for the linearized equation since when we take $(u, n, \phi)=0$
 that is when we linearize exactly on the approximate solution, then  we  get from the above result  the estimate
 $$  \big\| \big( \np,\up, \fp, \eps \nabla \fp)(t) \|_{L^2}^2 
   \leq C_{a}  \Big(  \big\| \big( \np_{0}, \up_{0}) \|_{L^2}^2  + \int_{0}^t \Big(\big\| (\eps^{-1}r_{\phi}, \pa_{t} r_{\phi},  r_{n}, r_u\big)\|_{L^2}^2+   \big\|\big( \np,\up, \fp, \eps \nabla \fp) \|_{L^2}^2\Big)\Big)$$
   and hence from the Gronwall inequality, we obtain 
  $$  \big\| \big( \np,\up, \fp, \eps \nabla \fp) (t) \|_{L^2}^2  \leq e^{ C_{a} t }\Big(  \big\| \big( \np_{0}, \up_{0}) \|_{L^2}^2  + \int_{0}^t \big\| (\eps^{-1}r_{\phi}, \pa_{t}r_{\phi},  r_{n}, r_u\big)\|_{L^2}^2\Big), \quad \forall t \in [0,T]$$ 
  which is 
 an  $L^2$ type  estimate  for which the growth rate is uniform in $\eps\in (0, 1]$.

\begin{proof}
In order to get this estimate, we shall use  a linearized version of the total energy of the system.
At first, let us collect a few useful estimates that we shall use for $(n_{a}, u_{a})$.
In the proof, we shall denote by $C_{a}$ a number which may change from line to line but which is uniformly bounded for
 $\eps \in (0,1]$ and $T \in (0, T_{0}]$ where $T_{0}$ is the interval of time on which the approximate solution is defined.
 Since  the leading boundary layer term of $u_{a}$ vanishes, we have
\begin{equation} \label{boundnau}
 \sup_{(0,T) \times \R^3_+} |u_{a}|  + | \na_{x,t} u_a | \le C_a, \quad C_a > 0. 
 \end{equation}
 For $n_{a}$, we have
 \begin{equation}
\label{boundna}
  \sup_{(0,T) \times \R^3_+} |n_{a}| +  | \na_{x_1,x_2,t} n_a | \le C_a, \quad \sup_{(0,T) \times \R^3_+} | \pa_3 n_a | \le \frac{1}{\e} C_a.
\end{equation}
To make more precise the last estimate, we observe that 
\begin{equation}
\label{mieux1}
\pa_3 n_a= \pa_3 \left(N_0 \left(t,y,\frac{x_3}{\e}\right)\right) + O(1) = \frac{1}{\e} \partial_{Z} N_0\left(t,y,\frac{x_3}{\e}\right) + O(1),
\end{equation}
where $Z$ stands for the fast variable $x_3/\eps$ and  hence, from the exponential decay of the boundary layer, we get that
\begin{equation}
\label{mieux2}  \sup_{t  \in (0,T), x \in \R^+_3} \frac{1}{\eps}\left| x_3 \pa_3 N^0\left(t,y,\frac{x_3}{\eps}\right)   \right|  \le   \sup_{t \in (0,T), y \in \R^2, z \in \R_+} \left| z \pa_z N^0(t,y,z) \right| \: \le \: C_a,
\end{equation}    
Note also  that $\phi_a$ shares similar bounds:
\begin{equation}
\label{boundnaphi_1}
  \sup_{(0,T) \times \R^3_+} |\phi_{a}| +  | \na_{x_1,x_2,t} \phi_a | \le C_a, \quad \sup_{(0,T) \times \R^3_+} | \pa_3 \phi_a | \le \frac{1}{\e} C_a.
\end{equation}
with
\begin{equation}
\label{mieux1phi}
\pa_3 \phi_a= \pa_3 \left(\Phi_0 \left(t,y,\frac{x_3}{\e}\right)\right) + O(1) = \frac{1}{\e} \partial_{Z} \Phi_0\left(t,y,\frac{x_3}{\e}\right) + O(1),
\end{equation}
and again, from the exponential decay of the boundary layer, 
\begin{equation}
\label{mieux2phi}  \sup_{t  \in (0,T), x \in \R^+_3} \frac{1}{\eps}\left| x_3 \pa_3 \Phi^0\left(t,y,\frac{x_3}{\eps}\right)   \right|  \le   \sup_{t \in (0,T), y \in \R^2, z \in \R_+} \left| z \pa_z N^0(t,y,z) \right| \: \le \: C_a,
\end{equation}    

 \bigskip
Let us now prove the energy estimate. First, multiplying the velocity equation by $(n_a + n) \, \up$, and performing standard manipulations, we obtain:
\begin{align}
 &  {{d}\over{dt}} \int_{\R^3_+}  (n_a + n) \frac{|\up|^2}{2} = \int_{\R^3_+} (n_a + n) \up \cdot \partial_t \up + \int_{\R^3_+} \pa_t (n_a + n) \frac{|\up|^2}{2} \\
 \label{momentum}
&=  I_1 + I_2 + I_3 - \int_{\mathbb{R}_{+}^3}(\up \cdot \nabla u_{a}) \cdot (n_{a}+ n) \up+ \int_{\R^3_+} \left( r_u \cdot ((n_a+n) \up) \right) + \int_{\R^3_+} \pa_t (n_a + n) \frac{|\up|^2}{2},  
\end{align}
where 
$$ I_1 \: :=  \: -  T^i  \int_{\R^3_+}  \left(\frac{\na \np}{n_a + n} - \frac{\na n_a}{n_a} \left( \frac{\np}{n_a + n} \right) \right) \cdot ((n_a + n) \up), \quad I_2  \: := \:  \int_{\R^3_+} \na \fp \cdot (n_a + n) \up,
 $$
 $$
 I_3 \: = \: - \int_{\R^3_+} [(u_a + u) \cdot \na \up] \cdot (n_a + n) \up .
 $$
The last three  terms at the r.h.s. of \eqref{momentum} can be  easily estimated by using \eqref{boundnau}, \eqref{boundna} and
 \eqref{hypmin}: 
\begin{equation} \label{firstterms}
\begin{split}
\int_{\R^3_+} \left( r_u \cdot ((n_a + n)\,  \up) \right) \: \le \:  C( C_{a}, M)   \| r_u (t)\|_{L^2(\R^3_+)}  \| \up(t) \|_{L^2(\R^3_+)}, \\
 \int_{\R^3_+} \pa_t (n_a + n) \frac{|\up|^2}{2} \leq C( C_a , M)\left(1+ \Vert  \partial_{t} \np \Vert_{ L^\infty}\right)\|\up(t)\|_{L^2(\R^3_{+})}^2,\\
 \int_{\mathbb{R}_{+}^3}(\up \cdot \nabla u_{a}) \cdot (n_{a}+ n) \up \le C( C_a , M) \|\up(t)\|_{L^2(\R^3_{+})}^2.
\end{split}
 \end{equation}
Let us turn to the treatment of $I_1$. Integrating by parts, we first have:
\begin{align*}
-   \int_{\R^3_+}  \left(\frac{\na \np}{n_a + n} \right) \cdot ((n_a + n) \up) &= \int_{\R^3_+} \frac{\np}{n_a + n} \div((n_a+n) \up) -  \int_{\R^3_+} \frac{\na (n_a + n)}{(n_a + n)^2} (n_a+n) \cdot \np \, \up
\end{align*}
and hence, we obtain that
\begin{equation}
\label{I1est1} I_{1}=  T^i   \int_{\R^3_+} \frac{\np}{n_a + n} \div((n_a+n) \up) +  T^i \int_{\R_{+}^3}
\Big( \frac{\na n_a}{n_a}  -  \frac{\na (n_a + n)}{n_a + n} \Big)\cdot \np \, \up:=I_{1}^1 + I_{1}^2.\end{equation}
To estimate $I_{1}^2$, we observe that
$$ \frac{\na n_a}{n_a}  -  \frac{\na (n_a + n)}{n_a + n}= { \nabla n_{a}\, n \,   - n_{a} \nabla n\, \over n_{a}(n_{a}+ n) }.$$
Consequently, we  can use \eqref{boundnau}, \eqref{boundna} and \eqref{hypmin} to get that
\begin{equation}
\label{I12}
|I_{1}^2 | \leq C(C_{a}, M) \big(  \|\nabla  n \|_{L^\infty} + \eps^{-1} \| n \|_{L^\infty}\big)  \|\up\|_{L^2(\R^3_{+})}\, \|\np\|_{L^2(\R^3_{+})}.
\end{equation}
 To estimate $I_{1}^1$,  we observe that we can write the first line of \eqref{linear} under the form
\begin{equation}
\label{nbis}
\partial_{t}\np + \div \big(  (u_{a}+ u) \np\big) + \div\big((n^a+ n) \up \big)= r_{n}.
\end{equation}
 By using this equation to express   $\div\big((n^a+ n) \up\big)$, we obtain
\begin{align}
{1 \over T^{i}} I_{1}^1 & =- \int_{\R^3_+} \frac{\np}{n_a + n} \left(  \pa_t \np +  \div \big((u_{a} + u)  \np\big)  -  r_n \right) \nonumber\\
 & = -  \pa_t \int_{\R^3_+} \frac{\np^2}{2 (n_a + n)} 
  + \: {1 \over 2}  \int_{\R^3_+}  \np^2  \Big( \pa_t + (u^a + u) \cdot \nabla \Big)\big( { 1 \over n^a+ n} \big) \nonumber\\
  \label{gentil}
  & \quad   -{1 \over 2 }  \int_{\R^3_+} \frac{\np^2 \, \div ( u^a + u)  }{n_a + n}      
 +  \int_{\R_3^+} r_n \frac{\np}{n_a + n}.
\end{align}
In the above expression, for the two last terms,  we  use  \eqref{hypmin} to get
$$   \Big| - {1 \over 2} \int_{\R^3_+} \frac{\np^2 \,\div(u_{a} + u ) }{n_a + n}     
 +  \int_{\R_3^+} r_n \frac{\np}{n_a + n} \Big|  \leq C(C_{a}, M)  \Big(\|r_{n}(t)\|_{L^2}  \, \|\np(t)\|_{L^2} + ( 1 + \|\nabla u \|_{L^\infty}) \,\|\np(t)\|_{L^2}^2 
  \Big).$$
  Next, we observe that
\begin{align*}
 &     \int_{\R^3_+}  \np^2  \Big( \pa_t + (u^a + u) \cdot \nabla \Big)\big( { 1 \over n^a+ n} \big)  \\
&  = -  \int_{\R^3_{+}} {\np^2 \over (n_{a}+ n)^2} \Big( \partial_{t}(n_{a} + n) +( u_{a}+ u) \cdot \nabla \big( n_{a}+ n)\Big).
  \end{align*}
  From the equation satisfied by $n_{a}$ and \eqref{boundnau}, we  get that
 $$ | \partial_{t} n_{a} + u_{a} \cdot \nabla  n_{a}  |\leq C_a.$$
  Also,   by using that $(u)_{3}$ vanishes on the boundary, we have $ |(u)_{3}| \leq x_{3} \|\nabla  u\|_{L^{\infty}_{T}}$
   and hence by using  \eqref{mieux2},  we also have
   \begin{equation}
   \label{cancel} \| u  \cdot \nabla n_{a}\|_{L^2} \leq  C_a\|\nabla  u\|_{L^\infty}.\end{equation}
  Consequently, we obtain that 
   \begin{align*}
 &  \Big|  \int_{\R^3_+}  \np^2  \Big( \pa_t + (u^a + u) \cdot \nabla \Big)\big( { 1 \over n^a+ n} \big)  \Big| \\
 &  \leq C(C_{a}, M) ( 1 +  \|\nabla_{t,x}n \|_{L^\infty} +   \|\nabla u\|_{L^\infty} )\|\np\|_{L^2}^2.
\end{align*}
We have thus proven that
\begin{align}
\label{I1}
& \Big| I_{1}  + {{d}\over{dt}}  \int_{\R^3_+}   \frac{ T^{i} \np^2}{2 (n_a + n) }\Big|    \\
\nonumber &\leq   C(C_a, M) \Big(  \big(  1 +  \| \nabla_{t,x}n \|_{L^{\infty}} +    \|\nabla u\|_{L^{\infty}}\big) \|\np\|_{L^2}^2 + 
\big( \eps^{-1} \| n\|_{L^\infty} +  \| \nabla n\|_{L^\infty} \big) 
 \|\np\|_{L^2}\, \|\up\|_{L^2} \Big)
   +\|r_{n}\|_{L^2}^2.
\end{align}

As regards $I_2$,  since
$$ I_{2} = -  \int_{\R^3_{+}} \fp \, \div \big( (n_{a}+ n)  \up \big), $$ 
 we use once again  \eqref{nbis} to write it as 
\begin{equation}
\label{I2-2}
I_2  =  \int_{\R^3_+}  \fp \, \pa_t \np \:  + \:  \int_{\R^3_+}  \fp  \:\div \big((u_{a}+ u)  \np \big)  - 
   \int_{\R^3_+}   r_n \fp.    
\end{equation} 

By differentiating with respect to time  the Poisson equation in \eqref{linear}, we can express $\pa_t n$ in terms of $\fp$, and substitute into the first  term at the r.h.s of \eqref{I2-2}: 
\begin{align*}
\int_{\R^3_+}  \fp \, \pa_t \np & =  \int_{\R^3_+}   \fp \, \left( \eps^2 \pa_t \Delta \fp  - \pa_t (e^{-\phi_{a}}( 1+ h(\phi)) \fp)  - \pa_t r_\phi \right) \\
 &= - \eps^2  \pa_t \int_{\R^3_+}  \frac{1}{2} |\na \fp|^2 \: - \: \eps \pa_t \int_{\R^3_+} \frac{1}{2} \fp^2 e^{-\phi_{a}} (1 + h(\phi)) \: - \:  \frac{1}{2} \int_{\R^3_+} \pa_t (e^{-\phi_{a}}(1+h(\phi))\big) |\fp|^2 \\
 & \quad  \: - \: \int_{\R^3_+} \fp \,  \pa_t r_\phi.
 \end{align*}
 
 For the last two terms, one has by using \eqref{boundnaphi_1} and \eqref{hypmin} the straightforward estimate:
\begin{equation}
\label{uno}
 \begin{aligned}
  - \frac{1}{2} \int_{\R^3_+} \pa_t (e^{-\phi_{a}}(1+h(\phi))) |\fp|^2  \:& - \: \int_{\R^3_+} \fp \,  \pa_t r_\phi \\
  &\leq C(C_a, M)
  \big(  1 +  \|\partial_{t} \phi\|_{L^\infty} \big) \| \fp \|_{L^2(\R^3_+)}^2 +  \|\partial_{t}r_{\fp}\|_{L^2(\R^3_{+})} \|\fp\|_{L^2(\R_{+}^3}).
 \end{aligned}
 \end{equation}
As regards the second term at the r.h.s. of \eqref{I2-2}, we  integrate by parts and use  the Poisson equation to  express $n$ in terms of $\fp$. We get  
\begin{align*}
 \int_{\R^3_+}  \fp \, \div(\np (u_a + u))  & =  - \int_{\R^3_+}   \na \fp \cdot (\np (u_a+ u)) \\
 & = - \int_{\R^3_+}   \na \fp \cdot (\eps^2 \Delta \fp (u_a + u))  \: + \:  \int_{\R^3_+}   \na \fp \cdot (e^{-\phi_{a}}(1+ h(\phi)) \fp (u_a+u)) \\
 & + \:     \int_{\R^3_+}   \na \fp \cdot (r_\phi (u_a + u))  = J_1 + J_2 + J_3.
\end{align*}
One has 
\begin{equation}
\label{duo}
J_2 =  \frac{1}{2} \int_{\R^3_+}   \na |\fp|^2 \cdot (e^{-\phi_{a}} (1+ h(\phi)) (u_a+ u)) \: = \: - \frac{1}{2}  \int_{\R^3_+}   |\fp|^2 \, \div(e^{-\phi_{a}}(1+ h(\phi))  (u_a + u)). 
\end{equation}
%Let $\chi = \chi(x_3) > 0$ a  smooth function  satisfying  $\chi(x_3) = x_3$ for small $x_3 > 0$, $\chi(x_3) = 1$ for large $x_3$. We write   
%$$ -\div(e^{-\phi_{a}-\phi}  u_a) = -\divy (e^{-\phi_{a}-\phi} u_{a,y}) -\pa_3 \left( \chi(x_3) \, e^{-\phi_{a}-\phi} \frac{u^0_3}{\chi(x_3)}\right)  -\pa_3 (e^{-\phi_{a}-\phi} (u_{a,3} - u^0_3)).   $$
%As $u^0_3$ is smooth and  cancels at $x_3 = 0$, $\frac{u^0_3}{\chi(x_3)}$ is smooth with bounded derivatives.  Taking into account that 
%$$  \sup_{t  \in (0,T), x \in \R^+_3} \frac{1}{\eps}\left| x_3 \pa_z \Phi^0\left(t,y,\frac{x_3}{\eps}\right)   \right|  \le   \sup_{t \in (0,T), y \in \R^2, z \in \R_+} \left| z \pa_z \Phi^0(t,y,z) \right| \: \le \: C_a,      $$
%respectively  that $u_{a,3} - u^0_3 = O(\eps)$,  we deduce that $ -\pa_3 \left( \chi(x_3) \, e^{-\phi_{a}} \frac{u^0_3}{\chi(x_3)}\right)  \le C_a$, respectively that $-\pa_3 (e^{-\phi_{a}} (u_{a,3} - u^0_3)) \le C_a$.  We end up with 
Once again, one has to be careful due to the boundary layer part of $\phi_{a}$. Nevertheless, proceeding as for estimate \eqref{cancel}, 
 we obtain that $(u_{a} + u) \cdot \nabla \phi_{a}$ is  uniformly bounded in $\eps$ and hence, we find:
$$
J_2 \: \le \: C(C_{a}, M) ( 1 + \| \nabla \phi\|_{L^\infty} +  \| \nabla u\|_{L^\infty})\|\fp\|_{L^2}^2. 
$$
Straightforwardly, thanks to \eqref{boundnau} one has also 
$$ J_3 \: \le \: \frac{1}{\e} (C_a \, + \, \|  u \|_{L^\infty})     \, \| r_\phi \|_{L^2(\R^3_+)} \,  \| \e \na \fp \|_{L^2(\R^3_+)}.  $$
Finally, we compute 
\begin{align}
\label{J1}
J_1 & = \: \eps^2 \int_{\R^3_+} ((\na \fp \cdot \na )  (u_a + u))\cdot \na \fp \: + \: \eps^2 \int_{\R^3_+}   \left( \na \frac{|\na \fp|^2}{2} \right) \cdot (u_a+u)  \\
       & \leq  \:  \eps^2 \int_{\R^3_+}  | \na \fp |^2 \, |\na (u_a+u)|  \: -  \: \eps^2\int_{\R^3_+}  \frac{|\na \fp|^2}{2}  \div (u_a+u) \:  \le \: 
       C(C_a, \|\na_x u \|_{L^\infty}) \, \eps^2 \int_{\R^3_+} | \na \fp |^2  \nonumber
\end{align}
where the last bound comes again from \eqref{boundnau}.

%Finally we bound the third term of \eqref{I2-2} as follows:
%$$
% \int_{\R^3_+}  \fp \, \div(n u) = -  \int_{\R^3_+}  \nabla \fp \cdot n u\:  \leq \: \frac{1}{\e} \Vert u \Vert_{L^\infty(\R^3_+)}   \Vert \e \na \fp \Vert_{L^2(\R^3_+)} \Vert n \Vert_{L^2(\R^3_+)} 
%$$

%
%

Combining the previous inequalities, we get the following bound: 
\begin{equation} \label{I2}
\begin{aligned}
&   I_{2} + {1 \over 2} {{d}\over{dt}}\Big( \int_{\R^3_{+}} \eps^2 |\nabla \fp|^2 + e^{\phi_{a}} (1 + h(\phi)) \fp^2\Big) \\ 
  &   \leq C (  C_a, M)\big( 1 + \|\nabla_{t,x} \phi\|_{L^\infty} \, + \,  \|\nabla_{x} u\|_{L^\infty}\big)   \big( \| \fp\|_{L^2}^2 + \eps ^2 \|\nabla \fp\|_{L^2}^2\big)  + \eps^{-2} \| r_{\phi}\|_{L^2}^2 + \|\partial_{t} r_{\phi}\|_{L^2}^2 \: + \: \| r_n \|_{L^2}^2. 
%I_2  \le  C_a \, \Bigg(  \eps^2  \int_{R^3_+} | \na \fp |^2  \: + \:   \left(1+  \Vert  \pa_t (e^{-\phi_{a}-\phi})\Vert_{L^\infty}\right) \int_{\R^3_+}  |\fp|^2   \\
%&\: + \: \frac{1}{\e} \Vert u \Vert_{L^\infty(\R^3_+)} \Vert \e \na \fp \Vert_{L^2(\R^3_+)} \Vert n \Vert_{L^2(\R^3_+)} \\  
%&\: + \:  \left(\|e^{-\phi} \|_{L^\infty(\R^3_+)} + \| \nabla \phi \, e^{-\phi} \|_{L^\infty(\R^3_+)} + \| r_n \|_{L^2(\R^3_+)} \, \: + \:  \| \pa_t r_\phi \|_{L^2(\R^3_+)} \: + \:  \| r_\phi \|_{H^1(\R^3_+)} \, \right) \\
%& \times \| \fp \|_{L^2(\R^3_+)} \Bigg). 
\end{aligned}
\end{equation} 
Finally, to estimate $I_3$ defined after  \eqref{momentum}, we write:
\begin{align*}
I_3 &= - \int_{\R^3_+} (n_a+n)(u_a + u) \cdot \na \frac{|\up|^2}{2} \: = \: \int_{\R^3_+}  \div( (n_a+n)(u_a + u)) \frac{|\up|^2}{2}.
\end{align*}
Relying once again on \eqref{boundnau} and proceeding like for \eqref{cancel}, we infer that:
\begin{equation}
\label{I3}
I_3 \leq C(C_a, M) \left(1 +  \| \nabla (n, u)\|_{L^\infty} \right) \Vert \up \Vert_{L^2(\R^3_+)}^2.
\end{equation}
Eventually, combining \eqref{momentum} with \eqref{firstterms}-\eqref{I1}-\eqref{I2}-\eqref{I3}, we obtain 
\begin{equation}
\label{esti}
\begin{aligned}
 {d}\over{dt} &\int_{\R^3_+}  \left( (n_a + n) \frac{|\up|^2}{2} \: + \: T^i \, \frac{\np^2}{2 (n_a+ n)} \: + \: \eps^2 \frac{|\na \fp|^2}{2} \: + \:    \frac{1}{2} |\fp|^2 e^{-\phi_{a}}( 1 + h(\phi)) \right) \\
 & \leq C( C_a, M) \big(1 +  \|\nabla_{t,x} (u, n, \phi)\|_{L^\infty} + \eps^{-1} \| n\|_{L^\infty}\big)\big( \|\up\|^2_{L^2} + \|\np\|^2_{L^2}+ \eps^2\|\nabla \fp \|^2_{L^2}
  + \|\fp\|_{L^2}^2 \big) \\
  &\quad  +  \eps^{-2} \|r_{\phi}\|_{L^2}^2 + \| \partial_{t} r_{\phi}\|_{L^2}^2+  \|r_{n}\|_{L^2}^2 + \|r_{u}\|_{L^2}^2.
\end{aligned}
\end{equation}
We end the proof by integrating in time and by using \eqref{hypmin}. 
\end{proof}

\subsection{Nonlinear stability}
  We shall now work on the nonlinear system \eqref{EPP} in order to get Theorem \ref{main}.
  Thanks to the well-posedness in $H^m$ for $m \geq 3$ of the system \eqref{EPP}, we can define
$$ T^\eps= \sup \big\{ T \in [0,  T_{0}], \quad   \forall t \in [0,T],  \|(n,u, \phi, \eps \nabla \phi)\|_{H^m_{\eps}(\R^3_{+})} \leq \eps^r,  \mbox{ and \eqref{apriori2} is verified} \big\}$$
where $r$ is chosen such that
\begin{equation}
\label{choixr}
 5/2<r<K
 \end{equation}
 and  the $H^m_{\eps}$ norm is defined by 
 $$ \|f\|_{H^m_{\eps}(\R^3_{+})} = \sum_{ | \alpha | \leq m} \eps^{|\alpha|} \| \partial_{x_{1}}^{\alpha_{1}} \partial_{x_{2}}^{\alpha_{2}} \partial_{x_{3}}^{\alpha_{3}} f \|_{L^2(\R^3_{+})}.$$  
   We shall also use the norms:
\begin{equation}
\label{defhronde}
\|f(t)\|_{\Hc^m_{co,\eps}(\R^3_{+})} = \sum_{ | \alpha | \leq m}  \| \ZZ_{0}^{\alpha_{0}} \ZZ_{1}^{\alpha_{1}}  \ZZ_{2}^{\alpha_{2}}
    \ZZ_{3}^{\alpha_{3}}f (t)\|_{L^2(\R^3_{+})}
   \end{equation}
    where the vector fields $\ZZ_{i}$ are defined by 
    \begin{equation}
    \label{defvectorfields}
     \ZZ_{0}=\eps \partial_t, \quad \ZZ_{i}= \eps \partial_{i}, \quad i= 1, \, 2, \quad \ZZ_{3}= \eps  {x_{3} \over 1+ x_{3}} \partial_{3}
     \end{equation}
 and 
    $$   \|f(t)\|_{\Hc^m_{\eps}(\R^3_{+})} = \sum_{ | \alpha | \leq m} \eps^{|\alpha|} \| \partial_{t}^{\alpha_{0}} \partial_{x_{1}}^{\alpha_{1}} \partial_{x_{2}}^{\alpha_{2}} \partial_{x_{3}}^{\alpha_{3}} f(t) \|_{L^2(\R^3_{+})}.$$
  For the sake of brevity, we will also use the following notation:
  $$ \ZZ^\a = \ZZ_{0}^{\alpha_{0}} \ZZ_{1}^{\alpha_{1}}  \ZZ_{2}^{\alpha_{2}}
    \ZZ_{3}^{\alpha_{3}}, \quad \text{for  } \a=(\a_0, \a_1,\a_2,\a_3).$$
  
   Finally, we set  
     $$ Q_{m}(t)=  \|(n,u, \phi, \eps\nabla \phi)(t)\|_{\Hc^m_{co, \eps}} + 
   \|\omega(t) \|_{\Hc^{m-1}_{\eps}}$$
   where we have set $\omega= \eps \, \curl u$ i.e:
   \begin{equation}
\label{defom}
\omega =  \eps  \begin{pmatrix}\pa_2 u_3 - \pa_3 u_2 \\ \pa_3 u_1 - \pa_1 u_3 \\ \pa_1 u_2 - \pa_2 u_1 \end{pmatrix}.
\end{equation} 
 
 We shall first prove that  $Q_{m}$ is the important quantity to control for the continuation of the solution and then we shall estimate $Q_{m}(t)$
  by using  Proposition \ref{propL2}.

 \subsubsection{Estimate of the  $\Hc^m_{\eps}$ norm}
  Let us recall a  classical estimate for products in   dimension $3$:
 \begin{equation}
 \label{prod0}
  \|u v\|_{L^2(\R^3_{+})} \lesssim  \|u \|_{H^{s_{1}}(\R^3_{+})} \|v\|_{H^{s_{2}}(\R^3_{+})},  \end{equation}
   with  $s_{1}+ s_{2} = 3/2$,  $s_{1}\neq 0$, $s_{2} \neq 0$. More generally, 
\begin{equation} \label{prod0bis}
   \| u_1 \dots u_{K} \|_{L^2(\R^3_{+})} \lesssim  \| u_1 \|_{H^{s_1}(\R^3_{+})} \dots  \| u_{K} \|_{H^{s_{K}}(\R^3_{+})}
 \end{equation}
 for $s_1 + \dots + s_{K} = \frac{3}{2} (K-1)$, $s_1 \neq 0, \dots, s_p \neq 0$.

% and
% \begin{equation}
% \label{prod0bis}
% \|u v\|_{H^s(\R^3_{+})} \lesssim \|u\|_{L^\infty}
%   \|v\|_{H^s} + \|v\|_{L^\infty}\|u\|_{H^s}, \quad s \geq 0
% \end{equation}
\medskip
As an application, we state
 \begin{lemma}
  \label{lemprodt}
  For any $u$, $v$ and $\eps \in (0,1]$, we have uniformly in $\eps$
  \begin{multline}
\label{prodt1}   \| (\eps\partial_{t})^k (uv) \|_{H^{l}_{\eps}(\R^3_{+})} 
 \lesssim \|u \|_{L^\infty} \| (\eps\partial_{t})^k v\|_{H^{l}_{\eps}(\mathbb{R}^3_{+})}
 + \|v\|_{L^\infty} \| (\eps\partial_{t})^k u\|_{H^{l}_{\eps}(\mathbb{R}^3_{+})}  
 \\  +  \eps^{-{3\over 2}}\left(  \|\langle \eps \pa_{t} \rangle^ku\|_{H^{l}_{\eps}(\R^3_{+})}  +   \|\langle \eps \pa_{t} \rangle^{k- 1}u\|_{H^{l+1}_{\eps}(\R^3_{+})}\right)\left(
 \|\langle \eps \pa_{t}\rangle^k v\|_{H^{l}_{\eps}(\R^3_{+})} +  \|\langle \eps \pa_{t}\rangle^{k- 1} v\|_{H^{l+1}_{\eps}(\R^3_{+})}\right)
 \end{multline}
 with the notation $\langle \eps \pa_{t} \rangle^k f = ( f, \eps \pa_{t}f, \cdots, (\eps \pa_{t})^k f)$ if $k \geq 0$, $0$ if $k<0$.
 Moreover for any smooth function  $F$ with $F(0) = 0$, we  have
   \begin{equation} \label{prodt2}
  \|(\eps \partial_{t})^k  F(u) \|_{H_\eps^l(\R^3_+)}   \,  \leq \, C   \: \big(
 \|\langle \eps \pa_{t}\rangle^k u\|_{H^{l}_{\eps}(\R^3_{+})} +  \|\langle \eps \pa_{t}\rangle^{k- 1} u\|_{H^{l+1}_{\eps}(\R^3_{+})}\big), \quad l \ge 1,
  \end{equation}
with 
$$ C = C\big[ \|u\|_{L^\infty(\R^3_+)},  \eps^{-{ 3 \over 2} } \|\langle \eps \pa_{t} \rangle^k u \|_{ H^l_{\eps}(\R^3_{+})} +  \eps^{-{ 3 \over 2} } \|\langle \eps \pa_{t} \rangle^{k-1} u \|_{ H_\eps^{l+1}(\R^3_{+})} \big].  $$

%  Finally, we also have that for $i=1,\,2,\,3$, 
%\begin{multline}
%\label{prodt3} 
%    \| (\eps \partial_{t})^k (u\, \eps  \partial_{i} v) \|_{H^l_{\eps}(\R^3_{+})}  
%    \lesssim \|(u, \eps \partial_{t} u, \eps \nabla  u)  \|_{L^\infty} \big( \| (\eps\partial_{t})^k v\|_{H^{l}_{\eps}(\mathbb{R}^3_{+})} +   \| (\eps\partial_{t})^{k-1} v\|_{H^{l + 1}_{\eps}(\mathbb{R}^3_{+})} \big) \\
%  \quad  \quad + \|\eps \nabla v\|_{L^\infty} \| (\eps\partial_{t})^k u\|_{H^{l}_{\eps}(\mathbb{R}^3_{+})}  +  \eps^{-{3\over 2}} \|u\|_{\Hc^{k+ l}_{\eps}(\R^3_{+})} 
% \|v\|_{\Hc^{k+ l}_{\eps}(\R^3_{+})}.
% \end{multline}
  \end{lemma}
   \begin{proof}
   Let us  first emphasize that a simple rescaling ($t'=t/\eps$, $x' = x/\eps$) allows to restrict to the case $\eps = 1$.
   
   \medskip 
  To prove \eqref{prodt1},  it suffices to use the Leibnitz formula. When all the  derivatives are on $u$ or on $v$, we estimate the other term in $L^\infty$.
   For the remaining terms  which are under the form $\partial_{t}^{k_{1}}\partial_{x}^{\alpha_{1}} u \, (\partial_{t})^{k_{2}} \partial_{x}^{\alpha_{2}} v$
    with $k_{1}+ |\alpha_{1}| \leq k + l  - 1$ and $k_{2}+  |\alpha_{2} | \leq k+ l  -1$, we use
     \eqref{prod0} to  write
  $$ \| \partial_{t}^{k_{1}} \partial_{x}^{\alpha_{1}} u \, \partial_{t}^{k_{2}} \partial_{x}^{\alpha_{2}} v\|_{L^2}
   \leq   \|  \partial_{t}^{k_{1}} \partial_{x}^{\alpha_{1}} u \|_{H^{s_{1}}(\R^3_{+})} \| 
 \partial_{t}^{k_{2}} \partial_{x}^{\alpha_{2}} v \|_{H^{s_{2}}(\R^3_{+})}$$
  with $s_{1} + s_{2}= 3/2$. Therefore, by  taking $s_{1}$ and $s_{2}$ smaller than $1$, we obtain 
    \eqref{prodt1}.  
    
    \medskip
The proof of  \eqref{prodt2} (still for $\eps = 1$) relies on the Faa Di Bruno formula for the quantity $\pa_t^k \pa_x^\alpha F(u)$, $|\alpha| \le l$, and on the product rule \eqref{prod0bis}. For brevity, we just treat the case $l=1$. 

If $| \alpha| = 0$, $\pa_t^k \pa_x^\alpha F(u) = \pa_t^k   F(u)$ can be decomposed thanks to the Faa Di Bruno formula  as a sum of terms  of the type 
$$ F^{(p)}(u) \prod_{i=1}^p \pa_t^{k_i} u, \quad \mbox{¬†with }¬†\quad \sum k_i \le k.  $$  
Either there is zero factor in the product (that is $k=0$), and we use the bound
$$ \| F(u)   \|_{L^2(\R^3_+)}  = \| F(u)  - F(0) \|_{L^2(\R^3_+)} \: \le \: C\left[\|u \|_{L^\infty}\right] \,  \| u \|_{L^2(\R^3_+)} $$
or there is one factor in the product, which leads to 
$$  \| F'(u) \pa_t^k u  \|_{L^2(\R^3_+)}  \: \le \: C\left[\|u \|_{L^\infty(\R^3_+)}\right] \|  \pa_t^k u  \|_{L^2(\R^3_+)}  $$
or there are several factors (which means $k_i < k$ for all $i$), and we have by  inequality \eqref{prod0bis} (with $s_1, \dots, s_{p-1}$ close to $3/2$ and $s_p = 1$)
$$  \|  F^{(p)}(u) \prod_{i=1}^p \pa_t^{k_i} u  \|_{L^2(\R^3_+)}    \: \le \: C\left[\| u \|_{L^\infty(\R^3_+)}, \|\langle \pa_{t}\rangle^{k-1} u\|_{H^{3/2}(\R^3_{+})}\right]  \|\langle \pa_{t}\rangle^{k-1} u\|_{H^1(\R^3_{+})} $$
 
 If $|\alpha| = 1$, for instance $\alpha = (1,0,0)$, $\pa_t^k \pa_x^\alpha F(u)$ can be decomposed as a sum of terms of the form
 $$ F^{(p+1)}(u) \prod_{i=1}^p \pa_t^{k'_i} u \: \left(  \pa_t^{k''} \pa_{x_1} u\right), \quad \mbox{¬†with }¬†\quad \sum k'_i  + k'' \le k. $$
 Either $k'' = k$, and we use the bound
 $$ \| F'(u) \pa_t^k   \pa_{x_1} u  \|_{L^2(\R^3_+)} \: \le \:   C\left[\| u \|_{L^\infty}\right] \,   \|\langle \pa_{t} \rangle^k u\|_{H^1(\R^3_{+})} $$
 or $k'' \le k-1$. In this latter case, we apply again  \eqref{prod0bis}: if one of the $k'_i$'s is $k$, we get  
 \begin{equation*} \| F^{(2)}(u) \:  \pa_t^k u \: \:  \pa_{x_1} u \|_{L^2(\R^3_+)}
 \le \: C\left[\| u \|_{L^\infty},     \|\langle  \pa_{t} \rangle^k u \|_{ H^1(\R^3_{+}) } \right]
 \, \|\langle \pa_{t}\rangle^{k- 1} u\|_{H^2(\R^3_{+})}
  \end{equation*}
  or all $k_i$'s are less than $k-1$, so that 
   $$ \|  F^{(p+1)}(u) \prod_{i=1}^p \pa_t^{k'_i} u \,  \: \left(  \pa_t^{k''} \pa_{x_1} u \right)\|_{L^2(\R^3_+)} \: \le \:  C\left[\| u \|_{L^\infty},   \|\langle \pa_{t}\rangle^{k- 1} u\|_{H^{3/2}(\R^3_{+})}\right] \, \|\langle \pa_{t}\rangle^{k- 1} u\|_{H^{2}(\R^3_{+})}.  $$
The last inequality comes from \eqref{prod0bis} with $s_1 = \dots = s_p$ close to $3/2$ and $s_{p+1} = 1$ (for the term $\pa_t^{k''} \pa_{x_1} u$). Combining above inequalities is enough to obtain \eqref{prodt2}.

 \end{proof}

 We shall first prove that by using the equation, we can estimate the $\Hc^m_{\eps}$ norm of the solution of \eqref{EPP}
  on $[0, T^\eps):$

 \begin{proposition}
 \label{propequiv1}
 For $m \geq 3$, we have for  every $\eps \in(0, 1]$,
  for  $t\in [0, T^\eps)$
     \begin{equation} \label{equiv1}
    \|(n,u, \phi, \eps \nabla \phi)(t) \|_{L^\infty(\R^3_{+})} \leq C \eps^{ r -3/2}, \quad    \| \nabla \big(n,u, \phi, \eps \nabla \phi\big)(t) \|_{L^\infty(\R^3_{+})} \leq  C\eps^{r-5/2}
   \end{equation} 
  for some $C>0$ independent of $\eps$ and
 \begin{equation}
 \label{Hcmeps}
 \|\big(n,u, \phi, \eps \nabla \phi\big)(t) \|_{\Hc^m_{\eps}(\R^3_{+})} \leq C[C_{a}, M] \big( \eps^{K+1} + \eps^r \big)
 \end{equation}
 where $C$ stands for a continuous non-decreasing function with respect to all  its arguments which does not depend on $\eps$.
 \end{proposition}
 
 \begin{proof}
The first set of  estimates  can be obtained by using the Sobolev inequality in dimension $3$ and the definition of $T^\eps$.

\medskip
To prove \eqref{Hcmeps}, we proceed by induction on the number $k$ of time derivatives. For $k=1$, we proceed as follows.  At first, by using the evolution equations on $n$ and $u$ in system \eqref{EPP}, we can compute $\eps \pa_t (n,u)$. We claim that, 
 for $t \in [0, T^\eps)$ 
 \begin{multline} \label{dtun} \| \eps \pa_{t}(n,u)(t) \|_{H^{m-1}_{\eps}(\R^3_{+})}
  \leq C\left[C_{a}, M,    \|(u,n)\|_{L^\infty(\R^3_{+})},  \|\eps \nabla(u,n)\|_{L^\infty(\R^3_{+})},  \eps^{-{3\over 2 }}  \|(u,n) \|_{H^{m}_{\eps}(\R^3_{+})}
   \right] \\ \left(  \|(u, n) \|_{H^{m}_{\eps}(\R^3_{+})}  \: + \:  \| \eps \nabla \phi \|_{H^{m-1}_{\eps}(\R^3_{+})}     + \eps^{K+1}\right).
    \end{multline}
 Indeed, the expression of $\eps \pa_t (n,u) $ involves four kinds of terms (besides the "easy ones" $\eps\na \phi$ and  $\eps^{K+1} R_{n,u}$): 

{\bf i)} terms that are linear in $\eps (\na n, \na u)$,   with coefficients depending on $u_a, n_a$. They are bounded by $C_a  \,    \| \eps \nabla(u, n) \|_{H^{m-1}_{\eps}(\R^3_{+})}$, where $C_a$ depends on the $L^\infty$ norms of $(\eps \na_x)^\a (u_a, n_a)$, with $|\a|\leq m-1$.
 
{\bf ii)}terms that are linear in $(n,u)$, with coefficients depending on $n_a, \eps \na n_a$, $\na u_a$. They are bounded by $C[C_a,M] \,  \|(u, n) \|_{H^{m-1}_{\eps}(\R^3_{+})}$.
 
{\bf iii)}  quadratic terms, involving products of $n$ or $u$ with $\eps \na n$ or $\eps \na u$. By \eqref{prodt1}, Lemma \ref{lemprodt}, we can bound them by 
\begin{multline*} 
C\left[   \|(u,n)\|_{L^\infty} + \|\eps \nabla(u,n)\|_{L^\infty(\R^3_{+})} +  \eps^{-{3\over 2 }}  \|(u,n) \|_{H^{m-1}_{\eps}(\R^3_{+})}
   \right]  \\
    \left(  \|(u, n) \|_{H^{m-1}_{\eps}(\R^3_{+})} +   \| \eps \nabla(u, n) \|_{H^{m-1}_{\eps}(\R^3_{+})} \right) 
    \end{multline*}

{\bf iv)} a fully nonlinear term, coming from the pressure in the Euler equation for the ions. It reads 
 $$ \eps \na F(n_a,n), \quad \mbox{¬†with }  F(n_a, n) \, := \, T^i \left(  \frac{\ln(n_a+ n)}{n_a + n}  -  \frac{\ln(n_a+ n)}{n_a} \right) $$
 In particular,   $F(n_a,0) = 0$. The nonlinear term can then be evaluated in $H^{m-1}_\eps$ using \eqref{prodt2} (with $k=0, l=m$).  It is bounded by 
\begin{equation*}
 C\left[M, C_a,  \|n\|_{L^\infty(\R^3_{+})},  \eps^{-{3\over 2 }}  \|n \|_{H^{m}_{\eps}(\R^3_{+})} \right] \\
  \| n \|_{H^{m}_{\eps}(\R^3_{+})}. 
 \end{equation*}

\medskip

 Combining the previous bounds yields \eqref{dtun}.  Therefore, by using \eqref{equiv1} and the definition of $T^\eps$, we  obtain
    \begin{equation}
    \label{dt1}  \| \eps \pa_{t}(n,u)(t) \|_{H^{m-1}_{\eps}(\R^3_{+})} \leq  C\left[C_{a}, M,  \eps^{r- 3/2}\right]\big( \eps^{K+1} + \eps^r \big).\end{equation}
    Moreover,  by applying $\eps \pa_{t}$ to the Poisson equation in \eqref{EPP}, we get that
    $$\big( \eps^2  \Delta    - e^{- (\phi_{a} + \phi)}\big) \eps \partial_{t} \phi= \eps  \partial_{t}n + \eps \pa_{t} \big( e^{-\phi_{a}}\big)
     \big(e^{-\phi} -1 \big) +  \eps^{K+1}  \eps\pa_{t} R_{\phi}, \quad
    \eps \pa_{t} \phi_{/x_{3}= 0}= 0$$
    and by applying next $(\eps \pa_x)^\alpha$,
  \begin{multline}
     \big( \eps^2  \Delta    - e^{- (\phi_{a} + \phi)}\big) (\eps \pa_x)^\alpha  \eps \partial_{t}  \phi \\ =  (\eps \pa_x)^\alpha \eps \partial_{t}  n + (\eps \pa_x)^\alpha G_1(\cdot, \phi) \:  + \: [ (\eps \pa_x)^\alpha ; G_2(\cdot, \phi)]   \eps \partial_{t} \phi  + \eps^{K+1} (\eps \pa_x)^\alpha   \eps \pa_{t}  R_{\phi}, 
 \end{multline}
  where $G_1(\cdot,\phi) = \eps \pa_{t} \big( e^{-\phi_{a}}\big)  \big(e^{-\phi} -1 \big)$, $\quad G_2(\cdot, \phi) =  e^{- (\phi_{a} + \phi)}$. 
From there, one can perform standard energy estimates, recursively on $|\alpha|$. The nonlinearities are handled thanks to   \eqref{prodt2}. This leads to
\begin{multline}
 \|\eps \pa_{t} ( \phi, \eps \nabla \phi )  \|_{H^{m-1}_{\eps}(\R^3_{+})}
    \leq  C\left[C_{a}, M, \|\phi\|_{L^\infty(\R^3_+)},\eps^{-3/2} \|\phi\|_{H^{m}_{\eps}(\R^3_{+})} \right] \\
     \big( \| \eps \pa_{t} n \|_{H^{m-1}_{\eps}(\R^3_{+})} +  \|\phi\|_{H^{m-1}_{\eps}(\R^3_{+})} + \eps^{K+1} \big).
 \end{multline}
    Consequently, by combining the last estimate and \eqref{dt1} and by using that $r>s$, we get that 
    $$ \big\|\eps \pa_{t} (n,u, \phi, \eps \nabla \phi )\|_{H^{m-1}_{\eps}(\R^3_{+})}
     \leq C[C_{a}, M]( \eps^{K+1} +  \eps^r).$$
     Now let us assume that we have proven that
    \begin{equation}
    \label{dtk}
    \| (\eps \pa_{t})^k (n,u, \phi , \eps \nabla \phi ) \|_{H^{m-k}_{\eps}(\R^3_{+})}
     \leq C[C_{a}, M]\big( \eps^{K+1} + \eps^r\big).
       \end{equation}
     By applying $\eps \, (\eps \pa_t )^{k}$ to the evolution equations for $n$ and $u$ in \eqref{EPP}, we obtain an expression for $(\eps \pa_t )^{k+1} (n,u)$. This expression invoves $k$ $\eps$-derivatives with respect to time of linear, quadratic and fully nonlinear terms. These terms are  evaluated  again thanks to Lemma \ref{lemprodt}.  We  get that for $k+1 \leq m$, $k \geq 1$
   \begin{align*} 
   & \| (\eps \pa_{t})^{k+1}(n,u)(t) \|_{H^{m-k-1}_{\eps}(\R^3_{+})}  \\
   & \leq  C\left[C_{a},M,    \|(u,n)\|_{L^\infty(\R^3_{+})},  \|\eps \nabla(u,n)\|_{L^\infty(\R^3_{+})}, 
  \eps^{-{3 \over 2 } } \sum_{j=0}^1 \| \langle \eps \pa_{t}\rangle^{k-j}(u,n) \|_{H^{m-k+j}_{\eps}(\R^3_{+})} 
   \right] \\
   &  \left(  \sum_{j=0}^1 \| \langle \eps \pa_{t}\rangle^{k-j}(u,n) \|_{H^{m-k+j}_{\eps}(\R^3_{+})}  
 + \|(\eps \pa_{t})^k \eps  \nabla \phi \|_{H^{m-k-1}_{\eps}(\R^3_{+})}
    + \eps^{K+1}\right)
    \end{align*}
     and hence from the induction assumption, we obtain
   \begin{equation}
   \label{dtk+11}  \| (\eps \pa_{t})^{k+1}(n,u)(t) \|_{H^{m-k-1}_{\eps}(\R^3_{+})} \leq  C[C_{a}, M]\big( \eps^{K+1} + \eps^r\big).
   \end{equation}
    Finally, we can use the Poisson equation to control $(\eps \pa_t)^{k+1} (\eps \na \phi, \phi)$. More precisely, applying $(\eps \pa_t)^{k+1}$ to the Poisson equation, we get 
    $$ \big( \eps^2  \Delta    - e^{- (\phi_{a} + \phi)}\big)( \eps \partial_{t})^{k+1} \phi= (\eps  \partial_{t})^{k+1}n + \mathcal{C}^k(\phi) +  \eps^{K+1} ( \eps\pa_{t})^{k+1} R_{\phi}, \quad
    (\eps \pa_{t})^{k+1} \phi_{/x_{3}= 0}= 0$$
    where $\mathcal{C}^k (\phi)$ is a commutator. Thanks to Lemma \ref{lemprodt} and the induction assumption, we have for this commutator the estimate
 \begin{align*}
&  \| \mathcal{C}^k (\phi) \|_{H^{m-k-1}_{\eps}(\R^3_{+})} \\ 
 & \leq C\Big[C_{a},  M,  \|\phi \|_{L^\infty(\R^3_+)}, \eps^{-3/2} \sum_{j=0}^1 \| \langle \eps \pa_{t}\rangle^{k-j} \phi \|_{H_{\eps}^{m-k+j}(\R^3_{+})} \Big]  
 \sum_{j=0}^1 \| \langle \eps \pa_{t}\rangle^{k-j} \phi \|_{H_{\eps}^{m-k+j}(\R^3_{+})}  \\ 
 &    \leq C[C_{a}, M] \big( \eps^{K+1} + \eps^r\big).
 \end{align*}
  Consequently,  the standard a priori estimates for the elliptic equation yield
   $$  \|(\eps \pa_{t})^{k+1}( \phi, \eps \nabla \phi) \|_{H^{m-k-1}_{\eps}(\R^3_{+})}
    \leq  C\left[C_{a}, M\right] \left( \| (\eps \pa_{t})^{k+1} n \|_{H^{m-k-1}_{\eps}(\R^3_{+})} + \eps^{K+1} + \eps^r \right)$$
     and hence we get by combining \eqref{dtk+11} and the last estimate that \eqref{dtk} is verified for $k$ changed in $k+1$.

\end{proof}

By using similar arguments as above, we  shall also get:

\begin{lemma}
\label{leminit}
We have for $\eps \in (0,1]$,  the estimate
$$ \|(n,u,\phi,  \eps \nabla \phi)(0) \|_{\Hc^m_{\eps}(\R^3_{+})} \leq C[C_{a}, M] \eps^{K}.$$

\end{lemma}

\begin{proof}
The proof follows the same lines as the previous Proposition. The main difference is that because of the choice of the initial data, 
 the estimate \eqref{dt1} is replaced by 
 $$  \| \eps \pa_{t}(n,u)(0) \|_{H^{m-1}_{\eps}(\R^3_{+})} \leq  C[C_{a}, \eps^{r- 3/2 }] \eps^{K}.$$
 The end of the induction can be performed in the same way.

\end{proof}

\subsubsection{Normal derivatives estimates}

The aim of this  subsection  is to prove that we can replace one normal derivative by tangential derivatives thanks to the equation
 and hence that  $ \|\eps\,\pa_{3}(n, u_{3})(t)\|_{\Hc^{m-1}_{ \eps}(\R^3_{+})}$  can be  estimated in terms of $Q_{m}$.
 We shall again need some product estimates:
   \begin{lemma}
  \label{lemprod}
  For any $u$, $v$ and $\eps \in (0,1]$, we have uniformly in $\eps$
  \begin{multline}
\label{prod1}   \| (\eps\partial_{3})^k \ZZ^\alpha (uv) \|_{L^2(\R^3_{+})} 
 \lesssim \|u \|_{L^\infty} \| (\eps\partial_{3})^k v\|_{\Hc^{|\alpha|}_{co,\eps}(\mathbb{R}^3_{+})}
 + \|v\|_{L^\infty} \| (\eps\partial_{3})^k u\|_{\Hc^{|\alpha|}_{co,\eps}(\mathbb{R}^3_{+})}  
 \\  +  \eps^{-{3\over 2}}\big( \|\langle \eps  \nabla \rangle^{k+1} u\|_{\Hc^{ |\alpha|- 1}_{co,\eps}(\R^3_{+})} 
 + \|\langle \eps \nabla \rangle^{k}\ u\|_{\Hc^{ |\alpha|}_{co,\eps}(\R^3_{+})}  \big)
 \big(\|\langle \eps \nabla \rangle^{k+1}v\|_{\Hc^{|\alpha |- 1}_{co,\eps}(\R^3_{+})}
 +\|\langle \eps \nabla \rangle^{k}v\|_{\Hc^{|\alpha |}_{co,\eps}(\R^3_{+})}\big). 
 \end{multline}
 Moreover for any smooth function  $F$ with $F(0) = 0$,   we  have  
 \begin{multline} \label{prod2}
  \|(\eps \partial_{3})^k \ZZ^\alpha F(u) \|_{L^2(\R^3_{+})}  \leq \, C_1 \, \|(\eps \partial_{3})^k  u \|_{\Hc^{|\alpha|}_{co,\eps}(\mathbb{R}^3_{+})}  \\
  + \:  C_2 \sum_{1 \le |\beta|\le 2} \left(  \| \langle \eps \pa_3 \rangle^{k-|\beta|}(\eps \na)^\beta u \|_{\Hc^{|\alpha|}_{co,\eps}(\mathbb{R}^3_{+})}  +  \| \langle \eps \pa_3 \rangle^{k+1-|\beta|} (\eps \na)^\beta u \|_{\Hc^{|\alpha|-1}_{co,\eps}(\mathbb{R}^3_{+})} \right)
    \end{multline}
   where  
   $$  C_1 =  C_1\left[ \|u\|_{L^\infty(\R^3_+)}\right], \quad 
   C_2  = C\left[ \|u\|_{L^\infty(\R^3_+)}, \eps^{-3/2} \| u \|_{\Hc^{k + | \alpha|}_{\eps}(\R^3_+)}\right]  $$
   are two  continuous, non-decreasing functions independent of $\eps$.   Moreover $C_2[\cdot, 0] = 0$.  
  \end{lemma}
 Note that the last statement of the lemma implies that $C_2$ is small when its second argument is small. This will be used in the sequel to absorb some error terms in the estimates.

 \begin{proof}
  To prove \eqref{prod1},  it suffices to use again  the Leibnitz formula. When all the  derivatives are on $u$ or on $v$, we estimate the other term in $L^\infty$.
   For the remaining terms  which are under the form $(\eps\partial_3)^{k_1} \ZZ^{\alpha_1} u \, (\eps\partial_3)^{k_2} \ZZ^{\alpha_2} v$
    with $k_{1}+ |\alpha_{1}| \leq k + | \alpha | - 1$ and $k_{2}+  |\alpha_{2} | \leq k+ | \alpha| -1$, we use
     \eqref{prod0} to  write
  $$ \| (\eps\partial_3)^{k_{1}} \ZZ^{\alpha_{1}} u \, (\eps\partial_3)^{k_{2}} \ZZ^{\alpha_{2}} v\|_{L^2}
   \leq \eps^{-3/2} \|  (\eps\partial_3)^{k_{1}} \ZZ^{\alpha_{1}} u \|_{H^{s_{1}}(\R^3_{+})} \| 
(\eps\partial_3)^{k_{2}} \ZZ^{\alpha_{2}} v \|_{H^{s_{2}}(\R^3_{+})}$$
  with $s_{1} + s_{2}= 3/2$ and we get the result  by  taking $s_{1}$ and $s_{2}$ smaller than $1$.
 
 \medskip
 The proof of \eqref{prod2} is very similar to the one of \eqref{prodt2}: broadly, $(\eps \pa_3)$  substitutes to  $\eps \na$, whereas $\ZZ$ substitutes to $\eps \pa_t$.  The full gradient terms $(\eps \na)^\beta$, $|\beta|=1, 2$, are still connected to the use of \eqref{prod0bis}, with $s_i \in [1,3/2]$.    
We leave the details to the reader.  
 \end{proof}
 
With this technical lemma in hand, we now prove the
 
\begin{proposition}
\label{propdn}
There exists $\eps_{0}$ such that for every $\eps \in (0, \eps_{0}]$, $t \in [0, T^\eps)$,  $0 \leq k \leq m$,
$$ \| (\eps \partial_{3})^k (n,u, \phi, \eps \nabla \phi ) (t)\|_{\Hc_{co,\eps}^{m-k}(\R^3_{+})}
\leq  C[C_{a},  M]\big( \eps^{K+1} +  Q_{m}(t)\big).$$
 
\end{proposition}
Note that we can reformulate the above Proposition into 
\begin{equation}
\label{HmQm}
 \| (n,u, \phi, \eps \nabla \phi)(t) \|_{\Hc^m_{\eps}(\R^3_{+})} \leq C[C_{a}, M] \big( \eps^{K+1} + Q_{m}(t)\big), \quad \forall t \in [0, T^\eps).
\end{equation}

\begin{proof}
 By using the definition of the vorticity, we already have:
$$ \|\eps \partial_{3}(u_{1}, u_{2}) \|_{\Hc^{m-1}_{co,\eps}} \lesssim \|\omega \|_{\Hc^{m-1}_{co, \eps}} + \|u\|_{\Hc^m_{co,\eps}} \lesssim Q_{m}(t).$$
Next, we notice that 
$$\|\eps \partial_{3} \na \phi \|_{\Hc^{m-1}_{co,\eps}} \lesssim  \|\eps \na \phi \|_{\Hc^m_{co,\eps}} + \| (\eps \pa_3)^2  \phi \|_{\Hc^{m-1}_{co,\eps}}. $$
By using the Poisson equation to express $\eps^2 \partial_{33}  \phi$, together with  Lemma \ref{lemprod} to control the nonlinearity,  we  get that
\begin{align*}
\| \eps \pa_{3} (\eps \nabla \phi)(t)\|_{\Hc^{m-1}_{co, \eps}} & \leq  C\left[C_{a}, M, \|\phi\|_{L^\infty(\R^3_+)}, \eps^{-3/2} \| \phi \|_{\Hc^{m-1}_{\eps}(\R^3_+)}\right]   \big( \eps^{K+1}+ Q_{m}(t)\big) \\
&  \leq  C\left[C_a, M, \eps^{r-3/2}\right]   \big( \eps^{K+1}+ Q_{m}(t)\big).
\end{align*}
Finally, by using the equations on $n$ and $u_3$ in  \eqref{EPP}, we get that %% Changé A_n, Daniel.
\begin{equation}
\label{systnorm}
\begin{aligned}
&  A_{n} \left( \begin{array}{ll} \partial_{3} u_{3} \\ \partial_{3} n  \end{array} \right)   = \\
& \left( \begin{array}{ll} \partial_{t}n  + (u_{a}+ u)_{1,2} \cdot \nabla_{1,2} n  + (1 - \chi)(u_a + u)_{3}\pa_{3} n + u \cdot \nabla n_{a} + n\, \div u_{a}  + (n+n_a) \na_{1,2} \cdot(u_1,u_2) - \eps^K R_{n} \\ \partial_{t} u_{3} + (u_{a}+ u)_{1,2} \cdot \nabla_{1,2}u_{3} +  (1 - \chi)(u_a + u)_{3}\pa_{3} u_3 - T^i { (\pa_{3} n_{a} n) \over n_{a} (n_{a}+ n) } + \pa_{3} \phi  - \eps^K R_{u} \end{array}\right)
\end{aligned}
 \end{equation}
 where $A_{n}$ is defined by 
 $$ A_{n}= - \left( \begin{array}{cc} n+ n_{a} & \chi(u_{a} + u)_3 \\ \chi(u_{a}+ u)_3 &   T^i(n^a + n)^{-1} \end{array} \right).$$ 
 Note that thanks to \eqref{apriori2}, we get that
 $A_{n}$ is invertible, its inverse being given by: %Rajouté explications, Daniel

 $$
 A_n^{-1}= \frac{1}{T^i - [\chi (u_a +u)_3]^2}\left( \begin{array}{cc}T^i(n^a + n)^{-1}& -\chi(u_{a} + u)_3 \\ -\chi(u_{a}+ u)_3 & n_a + n   \end{array} \right).
 $$
 Hence, we can use this system to estimate $ \|\eps\,\pa_{3}(n, u_{3})(t)\|_{\Hc^{m-1}_{co, \eps}(\R^3_{+})} $. Note also that  the field $ \eps (1- \chi) \partial_{3}$ is equivalent  to $\ZZ_{3}$. One can again decompose the terms into
 %multiplied by $\frac{T^i(n^a + n)^{-1}}{T^i - [\chi (u_a +u)_3]^2} $, $\frac{ n_a+n}{T^i - [\chi (u_a +u)_3]^2}$ or $\frac{ \chi(u_{a} + u)_3}{T^i - [\chi (u_a +u)_3]^2}$ 

{\bf i)} linear terms, whose coefficients depend on $u_a, \na u_a, n_a, \eps \na n_a$ and involve $(u,n)$ or $\ZZ_i (u,n)$.  
 
  They are bounded by  $\displaystyle  C[C_a,M]  \| (n,u) \|_{\Hc^m_{co, \eps}} $. 
    
{\bf ii)} nonlinear terms: they all involve  products of the type  $\displaystyle F(x,n,u) \,  \ZZ_i (n,u)$, where the  function $F$ satisfies   $F(x,0,0) = 0$. They can be estimated with  Lemma \ref{lemprod}  ($k=0$, $|\alpha|=m-1$).  

\medskip

We obtain 
 \begin{multline}
 \|\eps\,\pa_{3}(n, u_{3})(t)\|_{\Hc^{m-1}_{co, \eps}(\R^3_{+})} 
 \leq C_1 \left(\| (n,u) \|_{\Hc^m_{co,\eps}} + \| \eps \na \phi \|_{\Hc^{m-1}_{co, \eps}(\R^3_{+})} \: + \: \eps^{K+1} \right) \\
  \: + \: C_2 \left(   \| (n,u) \|_{\Hc^m_{co,\eps}}  \: + \:    \|(\eps \pa_3) (n,u) \|_{\Hc^{m-1}_{co,\eps}} \right)
 \end{multline} 
 with
 $$C_1 = C_1\left[ C_a, M,  \| (n,u) \|_{L^\infty(\R^3_+)},  \| \ZZ (n,u) \|_{L^\infty(\R^3_+)} \right] $$
 and with 
 $$ C_2 = C_2\left[ C_a, M,  \| (n,u) \|_{L^\infty(\R^3_+)},  \| \ZZ (n,u) \|_{L^\infty(\R^3_+)}, \eps^{-3/2} \| (n,u) \|_{\Hc_\eps^m(\R^3_+)} \right] $$
 that vanishes with its last argument. We used here the notation $\ZZ f = (\ZZ_{i}f)_{0 \leq i \leq 3}$. 
 
 Now, from the $L^\infty$ estimates and \eqref{Hcmeps} in Proposition  \ref{propequiv1} and from the definition of $T_\eps$, it follows that 
\begin{multline*}
 \|\eps\,\pa_{3}(n, u_{3})(t)\|_{\Hc^{m-1}_{co, \eps}(\R^3_{+})} 
 \leq C'_1[\eps^{r-3/2}] \left(\| (n,u) \|_{\Hc^m_{co,\eps}(\R^3_{+})} + \| \eps \na \phi \|_{\Hc^{m-1}_{co, \eps}(\R^3_{+})} \: + \: \eps^{K+1} \right) \\
  + \: C'_2[\eps^{r-3/2}] \left(   \| (n,u) \|_{\Hc^m_{co,\eps}}  \: + \:    \|(\eps \pa_3) (n,u) \|_{\Hc^{m-1}_{co,\eps}} \right) 
 \end{multline*} %%j'ai changé des s en 3/2, daniel.
 where $C'_2$  vanishes with its argument. As $r > 3/2$, we can absorb the last term at the r.h.s  in the l.h.s for $\eps$ small enough: this yields 
   \begin{equation}  \label{dn1} 
 \|\eps\,\pa_{3}(n, u_{3})(t)\|_{\Hc^{m-1}_{co, \eps}(\R^3_{+})} 
 \lesssim \left(\| (n,u) \|_{\Hc^m_{co,\eps}} + \| \eps \na \phi \|_{\Hc^{m-1}_{co, \eps}(\R^3_{+})} \: + \: \eps^{K+1} \right) \: \lesssim \: \eps^{K+1} + Q_m(t).  \end{equation}

We can then prove Proposition \ref{propdn} by induction on $k$.
 
The estimate for $k=1$ is thus  already proven.
Let us assume that it is proven for $j \leq k$. At first,  from the definition of the vorticity, we find that
 $$ \|(\eps \partial_{3})^{k+1}(u_{1}, u_{2}) \|_{\Hc^{m-(k+1)}_{co,\eps}} \lesssim \|(\eps\partial_{3})^{k}\omega \|_{\Hc^{m-(k+1)}_{co, \eps}} + 
  \|(\eps \partial_{3})^k u\|_{\Hc^{m-k}_{co,\eps}} $$
  and the right hand side is already estimated thanks to the definition of $Q_{m}$ and the induction assumption.
   In a similar way, by  using the Poisson equation, we get that 
   $$ \|(\eps \partial_{3})^{k+1}  \eps \nabla  \phi \|_{\Hc^{m-(k+1)}_{co,\eps}} 
    \leq C[C_a, M, \eps^{r-3/2}] \big( \eps^{K+1} +  \sum_{j=0}^k \| (\eps \partial_{3})^j (n, \phi, \eps \nabla \phi ) \|_{\Hc^{m-j}_{co,\eps}}\big)$$
    and  we conclude again by the induction assumption.  Finally, to estimate $\|(\eps \partial_{3})^{k+1} (n,u_{3})\|_{\Hc^{m-(k+1)}_{co,\eps}}$,
     we use again \eqref{systnorm},  Lemma \ref{lemprod} and the induction assumption.

\end{proof}

\subsubsection{Main energy estimate}\label{MEE}

 We shall finally estimate $Q_{m}$ by  carefully using the  system \eqref{EPP} and the linear stability estimate of  Proposition \ref{propL2}.
 
 Let us before  state a useful commutator estimate:
 \begin{lemma}
For $i=1, \, 2, \, 3$, we have the estimate:
\label{lemcom}
\begin{multline}\| [\ZZ^\alpha, f  ] \partial_{i}g \|_{L^2(\R^3_{+})} \lesssim  C\big[ \|(\nabla_{t,x}f, \partial_{i} g) \|_{L^\infty(\R^3_+)}, 
 \eps^{-{5\over 2}} \|f\|_{\Hc^{|\alpha|}_{\eps}(\R^3_{+})} \big] \\
 \big(  \|f\|_{\Hc^{|\alpha|}_{\eps}(\R^3_{+})} +  \|g\|_{\Hc^{|\alpha|}_{\eps}(\R^3_{+})}
 \big).\end{multline}
\end{lemma}

\begin{proof}
To get this lemma, it suffices to combine the Leibnitz formula and \eqref{prod0} as in the proof of Lemma \ref{lemprodt} and Lemma \ref{lemprod}.
The only term that we handle in a different way is when all the derivatives but one are on $\pa_{i}g$, in this case, we write
$$ \| \ZZ f  \,\ZZ^{|\alpha|-1} \partial_{i} g \|_{L^2(\R^3_{+})} \leq  \| \nabla _{t,x } f \|_{L^\infty(\R^3_{+})} \| \eps \ZZ^{|\alpha|-1} \partial_{i}g \|_{L^2(\R^3_{+})}
 \lesssim   \| \nabla _{t,x } f \|_{L^\infty(\R^3_{+})}  \|g\|_{\Hc^{|\alpha|}_{\eps}(\R^3_{+})}.$$

\end{proof}

Our main energy estimate for \eqref{EPP} is the following:

\begin{proposition}
\label{propQm}
There exists $\eps_{0}>0$ such that for every $\eps \in (0, \eps_{0})$, we have the estimate
$$Q_{m}(t)^2  \leq C(C_{a}, M) \big( \eps^{2K}+ T \eps^{2K}+ \int_{0}^t Q_{m}^2 \big), \quad \forall t \in [0, T^\eps).$$
 \end{proposition}

\begin{proof}

 We shall first estimate $\|(n,u, \phi, \eps \nabla \phi )\|_{\Hc^m_{co, \eps}(\R^3_{+})}.$
  To this end, we apply the operator $\ZZ^\a$ to the system \eqref{EPP}. We obtain:
\begin{equation} \label{linear2}
\left\{
\begin{aligned}
 \pa_t  \ZZ^\a n +  (u_{a}+ u )\cdot \nabla \ZZ^\a n + (n_{a}+ n) \div \ZZ^\a u  +   \ZZ^\a u \cdot \nabla( n_a + n) + \ZZ^\a n \, \div(u+u_{a}) &  \\ 
  =\mathcal{C}_{n} 
 + \eps^K \ZZ^\alpha R_{n} & , \\ 
 \pa_t \ZZ^\a u + (u_a + u ) \cdot \na \ZZ^\a u  + \ZZ^\a u  \cdot \na u_a + T^i   \left(\frac{\na \ZZ^\a n}{n_a + n} - \frac{\na n_a}{n_a} \left( \frac{\ZZ^\a n}{n_a + n} \right) \right)  \\
 =  \na \ZZ^\a \phi + \mathcal{C}_u + \eps^K \ZZ^\a R_{u} & , 
 \\
 \eps^2 \Delta \ZZ^\a \phi = \ZZ^\a n + e^{-\phi_a} \ZZ^\a \phi( 1  + h)  + \mathcal{C}_\phi+ \eps^{K+1} \ZZ^\a R_{\phi}& .
\end{aligned}
\right.
\end{equation}
One has $h = h_0$  for $|\alpha| = 0$ and $h=h_1$ for $|\alpha| \ge 1$. The functions $\mathcal{C}_n,\mathcal{C}_u$ and $\mathcal{C}_\phi$ are remainders due to commutators. 
One can observe that this corresponds to the ``abstract'' system that we have studied in Proposition \ref{propL2}.
We shall  now  estimate the remainders  in order to be able to apply the stability estimate \eqref{esti}.

\medskip
We first claim that 
\begin{equation}
  \label{estcom1}
 \|(\mathcal{C}_{n}, \mathcal{C}_{u}) \|_{L^2} \leq C[C_{a},  M,  \eps^{r -s- 1}] \|(n,u, \phi, \eps \nabla \phi) \|_{\Hc^m_{\eps}(\R^3_{+})}
  \leq C[C_{a}, M] \big( \eps^{K+1}+ Q_{m}(t) \big). 
  \end{equation}
This estimate is in the same spirit as the previous ones.  One must  distinguish between  linear terms  and  nonlinear terms, the latter being controlled thanks to the product  estimates of the previous lemma. For brevity, we only point out two typical terms: 
\begin{itemize}
\item a typical linear term in $\mathcal{C}_{n}$ is $[\na n_a ; \ZZ^\alpha] \cdot  u$. It is {\it a priori} singular (because of the boundary layer), but noticing that 
$$\| [\ZZ, \nabla n_{a}] f \|_{L^2} \leq C_{a}  \|f\|_{L^2} $$
gives the bound $\| [\na n_a, \ZZ^\alpha] \cdot  u \|_{L^2(\R^3_+)} \le C_a \| u \|_{\Hc_\eps^{m}(\R^3_+)} \le C[C_a,M] \left( \eps^{K+1} + Q_m(t) \right)$. 
\item a typical nonlinear term in  $\mathcal{C}_u$ is the commutator
 \begin{equation*} 
 T^i \, [F(n_a,n)  \na ;  \ZZ^\alpha ] n  \: = \:    T^i[F(n_a,n) ;  \ZZ^\alpha ] \na n \: + \: T^i F(n_a,n)  [ \ZZ^\alpha ; \na ] n   
 \end{equation*}
 with $ F(n_a,n) := \frac{1}{n_a+n} - \frac{1}{n_a}$.  
 Clearly 
$$  \|  F(n_a,n) [ \ZZ^\alpha ; \na ] n \|_{L^2(\R^3_+)} \: \le \: \: C_M  \| n \|_{\Hc_\eps^{|\alpha|}(\R^3_+)} $$
 As regards the first term at the r.h.s., we use the commutator estimate of Lemma \ref{lemcom}:  
 \begin{multline*}
\| [F(n_a,n) ;  \ZZ^\alpha ] \na n \|_{L^2(\R^3_+)}  \le C\left[\| \na_{t,x} F(n_a,n), \na_x n \|_{L^\infty(\R^3_+)}, \eps^{-5/2}  \| F(n_a,n) \|_{\Hc^m_{\eps}(\R^3_{+})} \right] \\
 \left( \| F(n_a,n) \|_{\Hc^m_{\eps}(\R^3_{+})} \: + \: \| n \|_{\Hc^m_{\eps}(\R^3_{+})} \right).
\end{multline*}
The product estimate \eqref{prodt2} implies that 
$$ \| F(n_a,n) \|_{\Hc^m_{\eps}(\R^3_{+})}  \: \le \: C\left[C_a, M, \| n \|_{L^\infty(\R^3_+)}, \eps^{-3/2}  \| n \|_{\Hc^m_{\eps}(\R^3_{+})} \right]  \| n \|_{\Hc^m_{\eps}(\R^3_{+})},
$$
so that eventually 
$$ \| T^i \, [F(n_a,n)  \na ;  \ZZ^\alpha ] n \|_{L^2(\R^3_+)} \: \le \: C[C_a, M, \eps^{r-5/2}]  \| n \|_{\Hc_\eps^{m}(\R^3_+)}  \le C[C_a,M] \left( \eps^{K+1} + Q_m(t) \right). $$
\end{itemize}
The treatment of the other terms is left to the reader. 

\medskip
 For the commutator $\mathcal{C}_{\phi}$, we  get that
  $$
 \|  \mathcal{C}_{\phi} \|_{L^2} \leq C[C_{a},   \eps^{r-3/2}] \eps \| \phi \|_{\Hc^{m-1}_{\eps}(\mathbb{R}^3_{+})}
   + \eps \left( \| \eps \nabla \phi \|_{\Hc^m_{co,\eps}(\R^3_{+})} + \| \eps \nabla \phi \|_{\Hc^{m-1}_{\eps}(\R^3_{+})} \right).$$
  The first term at the r.h.s corresponds to the commutator with the semilinear term in the Poisson equation. The second one bounds the commutator $[\eps^2 \Delta ; \ZZ^\alpha] \phi$. 
  We use  the  fact that
   $$  [\ZZ_{3}, \partial_{3}] = - \partial_{3} \left( { x_{3} \over 1+ x_{3} } \right) \eps \partial_{3}.$$
   Consequently, by using again Proposition \ref{propequiv1} and Proposition \ref{propdn}, we get that
   \begin{equation}
   \label{estcom2}
 \eps^{-1} \|  \mathcal{C}_{\phi} \|_{L^2} \leq  C(C_{a},  M,  \eps^{r-3/2}) \big( \eps^{K+1} + Q_{m}\big).
   \end{equation}
 By using similar arguments, we also get that
 $$  \| \partial_{t}  \mathcal{C}_{\phi} \|_{L^2} \leq C(C_{a}, \eps^{r-3/2})  \| \phi \|_{\Hc^{m}_{\eps}(\mathbb{R}^3_{+})}
  +    \|\eps^2  \nabla^2 \phi \|_{\Hc^{m}_{co, \eps}(\R^3_{+})} + \| \eps \na \phi \|_{\Hc^{m}_{\eps}(\mathbb{R}^3_{+})}.$$
  To estimate the second term, we use the elliptic regularity for the Poisson equation in \eqref{linear2}. This yields
  $$  \eps^2 \| \nabla^2 \phi \|_{\Hc^{m}_{co,\eps}(\R^3_{+})} \lesssim  \|n \|_{\Hc^m_{\eps}} +  C[C_{a},  \eps^{r-3/2}]  \| \phi \|_{\Hc^{m}_{\eps}(\mathbb{R}^3_{+})} + \| \mathcal{C}_{\phi}\|_{L^2} + \eps^{K+1}$$
  and hence, we find that
  \begin{equation}
  \label{estcom3}
    \| \partial_{t}  \mathcal{C}_{\phi} \|_{L^2} \leq C[C_{a}, M,   \eps^{r-3/2}]\big( \eps^{K+1} +  Q_{m} \big).
  \end{equation}
Note that thanks to Proposition \ref{propequiv1}, we have that on $[0, T^\eps)$, 
$$ \| (n,u, \phi) \|_{L^\infty(\R^3_{+})} \leq C\eps^{r- 3/2} $$
and hence that the assumption \ref{hypmin}  in Proposition \ref{propL2} is  matched on $[0, T^\eps)$
for $\eps$ sufficiently small. Consequently, we can use Proposition \ref{propL2} and Lemma \ref{leminit}  for the system \eqref{linear2} to get 
 \begin{align*}
 & \big\| \big( n,u, \phi, \eps \nabla \phi) (t) \big\|_{\Hc^m_{co,\eps}(\R^3_{+})}^2 
   \leq   C[M]\Big(   \eps^{2K}  + \int_{0}^t \big\| (\eps^{-1}r_{\phi}, \pa_{t} r_{\phi},  r_{n}, r_u\big)\big\|_{L^2(\R^3_{+})}^2 \Big) \\
     & +    C[C_{a}, M]   \int_{0}^t  \big( 1 +   \big\| \nabla_{t,x}\big( n,u,\phi)\big\|_{L^\infty(\R^3_{+})}  +\eps^{-1} \| n\|_{L^\infty(\R^3_{+})}\big)  \big\| \big( n,u, \phi, \eps \nabla \phi\big)) \big\|_{\Hc^m_{co,\eps}(\R^3_{+})}^2 \Big)
  \end{align*}
  where we have set
 $$ r_{n}= \mathcal{C}_{n} + \eps^{K} \ZZ^\alpha R_{n}, \quad  r_{u}= \mathcal{C}_{u} +  \eps^K\ZZ^\alpha R_{u}, \quad 
  r_{\phi}= \mathcal{C}_{\phi} + \eps^{K+1} \ZZ^\alpha R_{\phi}.$$
  Consequently, by using the estimates \eqref{estirestes}, \eqref{estcom1}, \eqref{estcom2},  \eqref{estcom3} and Proposition \ref{propequiv1}, we get that on $[0,T_\eps]$
   \begin{align}
   \label{estQm1}
  \big\| \big( n,u, \phi, \eps \nabla \phi) (t) \big\|_{\Hc^m_{co,\eps}(\R^3_{+})}^2 
  &  \leq   C(C_{a}, M, \eps^{r-5/2})\Big(   \eps^{2K} +  \eps^{2 K}  t  + \int_{0}^t  Q_{m}(t)^2 \Big).
%  &   \leq  C(C_{a}, M)\big(   \eps^{2K+2} + T \eps^{2 K}  +  \int_{0}^t Q_{m}(t)^2  \big)
    \end{align}

\bigskip

It remains to estimate $\| \omega\|_{\Hc^{m-1}_{\eps}(\R^3_{+})}$ i.e. the second term in the definition of  $Q_{m}$.
By applying the operator $\eps \nabla \times$ to the second equation  in the system \eqref{EPP}, we find that 
\begin{equation}
\label{omega}
\partial_{t} \omega + (u_{a}+ u ) \cdot \nabla \omega  =  \eps  \mathcal{P}(u_{a},  \nabla u_{a})\big( \nabla u,  \nabla u) - \eps u \cdot \na \omega_a + \eps^K  R_{\omega}
\end{equation}
where $\mathcal{P}(u_{a}, \nabla u_{a})$ is a polynomial of  degree less than two whose coefficients  depend only on $u_{a}$ and
 $\nabla u_{a}$ and thus are uniformly bounded. Note that we have used that
 $$ \nabla \times \Big(  {\nabla n \over n_{a}+ n} - {\nabla n_{a}\over n_{a}} \Big( {n \over n_{a}+ n}\Big) \Big)
  = \nabla \times \nabla  \big( \log(n_{a}+ n) - \log n_{a}\big)=0.$$
  By applying the operator $(\eps \pa)^\alpha= (\eps \pa_{t, x})^\alpha$  to  \eqref{omega}, we find that
  \begin{equation}
  \label{domega} \partial_{t} (\eps \pa)^\alpha \omega + (u_{a}+ u ) \cdot \nabla(\eps \pa)^\alpha \omega  =\mathcal{C}_{\omega}
  \end{equation}
  where by using Lemma \ref{lemprodt}, we have
  $$ \| \mathcal{C}_{\omega}\|_{L^2(\R^3_{+})} \leq C[C_{a}, M, \eps^{r-5/2} ] \big(  \|u\|_{\Hc^m_{\eps}(\R^3_{+})}
   + \|[(\eps \pa)^\alpha ;  u \cdot \nabla ] \omega \|_{L^2(\R^3_{+})} + \eps^K\big).$$
   To estimate the commutator, we use the following variant of Lemma \ref{lemcom}: we first  write
    since $\| \omega \|_{\Hc^{m-1}_{\eps}} \lesssim  \|u\|_{\Hc^{m}_{\eps}}$ that
   $$  \|[(\eps \pa)^\alpha ;  u \cdot \nabla ] \omega \|_{L^2(\R^3_{+})} \leq  \|\na_{t,x}u\|_{L^\infty(\R^3_+)} \| u  \|_{\Hc^{m}_\eps(\R^3_{+})} + \| (\eps \pa)^\alpha u  \, \cdot \nabla \omega \|_{L^2(\R^3_{+})}
    + \eps^{-{5 \over 2} } \|u\|_{\Hc^m_\eps(\R^3_{+})}^2$$
    and to estimate the second term, we use also  \eqref{prod0}, to get  since $| \alpha | \leq m-1$, $m \geq 3$  that 
    $$   \| (\eps \pa)^\alpha u  \, \cdot \nabla \omega \| \lesssim  \eps^{- {3 \over 2} } \| u\|_{\Hc^m_{\eps}(\R^3_{+})} \| \nabla \omega \|_{H^1_{\eps}(\R^3_{+})}
   \lesssim \eps^{-{5 \over 2} } \|u \|_{\Hc^m_{\eps}(\R^3_{+})}^2.$$
   We thus  get  the estimate 
   $$  \| \mathcal{C}_{\omega}\|_{L^2(\R^3_{+})} \leq C[C_{a}, M, \eps^{r-5/2}]  \|u\|_{\Hc^m_{\eps}(\R^3_{+})}.$$
   Consequently, from a standard $L^2$ type energy estimate on the equation \eqref{domega} (let us recall that $(u_{a}+ u)_{3}$ vanishes on the boundary),
    we get that
    $$ {d \over dt } \| \omega \|_{\Hc^{m-1}_{\eps}(\R^3_{+})}^2 
     \leq C[C_{a}, M] \big(  \eps^{2K} + \int_{0}^t \|u \|_{\Hc^m_{\eps}(\R^3_{+})}^2\big)$$
     and hence by using Lemma \ref{leminit} and Proposition \ref{propdn}, we infer that for $t \in [0, T^\eps)$, 
     \begin{equation}
     \label{estomega}
     \| \omega \|_{\Hc^{m-1}_{\eps}(\R^3_{+})}^2  \leq C[C_{a}, M] \big( \eps^{2K}+  \eps^{2K} t + \int_{0}^t Q_{m}^2 \big).
     \end{equation}
     To end the proof of Proposition \ref{propQm}, it suffices to combine  the estimates \eqref{estomega} and \eqref{estQm1}.
     
     \end{proof}

     \subsubsection{Proof of Theorem \ref{main}}
     We can now easily conclude the proof  of  Theorem \ref{main}.
     By using Proposition \ref{propQm} and the Gronwall inequality, we get that
     $$  Q_{m}(t) \leq C[C_{a}, M] \eps^{K} e^{C[C_{a}, M]t}, \quad \forall t \in [0, T^\eps).$$
  This yields   thanks to \eqref{HmQm},
  $$ \|(n,u, \phi, \eps \nabla \phi)(t) \|_{H^m_{\eps}(\R^3_{+})} \leq C[C_{a}, M] \big(  e^{C[C_{a}, M]t} +\eps \big) \eps^K, \quad  \forall t \in [0, T^\eps).$$
   Note that we also have  by Sobolev embedding
   $$ \|(n, \phi)\|_{L^\infty(\R^3_{+})} \leq \eps^{K- 3/2} \big(  e^{C[C_{a}, M] t} +\eps \big), \quad \forall t \in [0, T^\eps)$$
    and 
    $$\big\| \chi(x_{3}) u_{3}\big\|_{L^\infty(\R^3_{+})} \leq  C[\delta] \eps^{K- 5/2} \big(  e^{C[C_{a}, M]t} +\eps \big), \quad \forall t \in [0, T^\eps).$$
     In view of the definition of $T^\eps$,  we obtain that for $\eps$ sufficiently small  $T^\eps\geq T_{0}$ and that
     $$ \|(n,u, \phi, \eps \nabla \phi)(t) \|_{H^m_{\eps}(\R^3_{+})} \leq C[C_{a}, M]\big(  e^{C[C_{a}, M]t} +\eps \big) \eps^K, \quad  \forall t \in [0, T_{0}].$$
     This ends the proof of Theorem \ref{main}.

\section{Further remarks}\label{further}

\subsection{Generalization to general smooth domains}\label{ouvertregulier}

In this paragraph, we briefly explain how the work which has been done for the half-space $\R^3_+$ can be adapted to handle the case of a smooth open set $\Omega \subset \R^3_+$ whose boundary $\partial \Omega$ is an orientable compact manifold.

\medskip

{\bf i)} {\em Derivation of boundary layers.}

The principle is to apply the tubular neighborhood theorem in $\R^3$; then, $\pa \Omega$ being orientable, there exist a  smooth function $\varphi: \, \pa \Omega \rightarrow \R$ (normalized to have $|\na \varphi|=1$) and a neighborhood $U:= \{x \in \Omega, \varphi<\eta\}$ of $\pa \Omega$ in $\Omega$ such that for any $x \in U$, there is one and only one $(\tilde x,y) \in \pa\Omega \times \R^{+*}$, such that
$$
x \, = \tilde x \: + \:  y \,  n_{\tilde x}.
$$
where $ n_{\tilde x}$ is the inward-pointing normal at $\tilde x$.

Then we can look for approximate solutions of the form:
\begin{equation} \label{ansatz2}
\begin{aligned}
(n^\eps_{app},u^\eps_{app}, \phi^\eps_{app})    & \: = \:  \sum_{i=0}^K \eps^i \left( n^i(t,x), u^i(t,x), \phi^i(t,x) \right) \\  & \: + \: \sum_{i=0}^K \eps^i \left( N^i\left(t, \tilde x, \frac{y}{\eps}\right), U^i\left(t, \tilde x, \frac{y}{\eps}\right), \Phi^i\left(t, \tilde x, \frac{y}{\eps}\right)\right).
\end{aligned}
\end{equation}
and the construction of Section \ref{construction} remains the same.

\medskip

{\bf ii)}  {\emph{Stability of the boundary layers}.}

The $L^2$ stability estimate that is based only on integration by parts  is still true. For higher order estimates, 
one can modify the definition of the conormal  Sobolev spaces in the usual way by considering  a finite set  of 
  generators of  vector fields  tangent  to the boundary.

\subsection{Generalization to non homogeneous conditions for $u$}\label{other}

Let us now explain how we could generalize our work to the following boundary condition on $u$:
\begin{equation}
u_3\vert_{x_3=0}= u_b.
\end{equation}

We shall impose a subsonic outflow condition $-\sqrt{T^i}<u_b<0$. One can readily check that the boundary $\{x_3=0\}$ is \emph{noncharacteristic} in this setting and that this boundary condition is maximal dissipative.
From the physical point of view, this boundary condition  means that some plasma is somehow absorbed at the boundary.

\medskip
Let us once again explain what are the main changes in each principal part of the proof. Details are left to the reader.

{\bf i)} {\emph{Derivation of boundary layers}.}

Considering the derivation of the approximation, computations get more tedious, but the results remain roughly unchanged. In particular, we can check that there is still no boundary layer term for the velocity at leading order, which means with the same notations as before, that $U^0=0$. The principle to show this fact is to obtain closed equations on $U^0_y$ and $U^0_3$ which have the trivial solution $0$, as it was done for $U^0_y$ in Section \ref{construction}.

\medskip

{\bf ii)} {\emph{Stability of the boundary layers}.}

The main part which is modified due to the new boundary condition is the proof  of Proposition \ref{propL2}, in which new boundary contributions appear since $(u_a + u)_3\vert_{x_3=0} \neq 0$: we have to treat them carefully.
 Nevertheless, we note that  since the boundary condition for the original problem  is dissipative, the boundary terms when we integrate by parts
 in the transport terms have the  ``good sign''.  For example, 
one typical new harmless term comes for instance in \eqref{gentil}. By integration by parts for the term
$$
-\int_{\R^3_+} \frac{\na (\np^2)  \cdot ( u^a + u)  }{n_a + n}   
$$
the new term that appears is $-\int_{\R^2} \frac{(\np^2)   ( u^a + u)\cdot(-e_3)  }{n_a + n}\vert_{x_3=0}$ and one can observe that $( u^a + u)_3\vert_{x_3=0} \leq 0$, so that this term has indeed the good sign.

One more significant term comes from the computation of $J_1$ in \eqref{J1}. One can check that a first integration by part yields the term $\mathcal{B} := \int_{\R^2} |{\na \fp}|^2 (u_a + u)_3\vert_{x_3=0}$ and a second integration by parts yields the term $-\frac{1}{2} \mathcal{B}$, so that only the sum of these two terms, $\frac{1}{2} \mathcal{B}$, has the good sign.

Furthermore, some terms also have to be estimated more carefully. In \eqref{uno} and \eqref{duo}, the two terms
$$
  - \frac{1}{2} \int_{\R^3_+} \pa_t (e^{-\phi_{a}}(1+h(\phi))) |\fp|^2  - \frac{1}{2}  \int_{\R^3_+}   |\fp|^2 \, \div(e^{-\phi_{a}}(1+ h(\phi))  (u_a + u))
$$
can not be estimated separately as before. They can be treated using the equation satisfied by $e^{-\phi_a}$ (that is the transport equation satisfied by $n_a$, up to some error term).

The end of the proof of Theorem \ref{main} is unchanged, except for the last part which gets actually simpler. Indeed, the boundary $\{x_3=0\}$ is now noncharacteristic and therefore we can estimate all normal derivatives using directly the equations on $(n,u)$, without having to introduce the vorticity.

\medskip

 When $0<u_b<\sqrt{T^i}$ the boundary $\{x_3=0\}$ is still noncharacteristic (this corresponds to some \emph{inflow} condition, which means that some plasma is injected at the boundary) but 
 we still  need to add  two (in dimension $3$)  boundary conditions in order to get a well-posed problem.
  There are many possibilities and some of them do not yield a  dissipative problem. In this case, 
  our approach to the $L^2$ stability estimate of Proposition \ref{propL2} would  break down. We leave this case for future work, as well as the case of supersonic outflow condition, relevant to plasma sheath layer problems.

\subsection{The relative entropy method}

In this paragraph, we focus on the one-dimensional case.
First, we recall that Cordier and Peng constructed global entropy solutions to the same isothermal Euler-Poisson system as the one studied in this paper, posed in the whole line $\R$ (see \cite{CP}). Although it has never been done, it seems reasonable to believe in the existence of some global entropy solution in the half line $\R^+$, at least for the simplest boundary condition $\phi_b =0$ (and non-penetration, that is $u_3\vert_{x_3=0}=0$).

\emph{Assuming} the existence of such solutions with weak regularity, one could study the quasineutral limit, using the so-called relative entropy method which was introduced by Brenier in \cite{Brenier}. This strategy was previously briefly evoked in \cite{HK}. It consists in considering the following modulated {\emph{nonlinear}} energy:
\begin{equation}
\begin{split}
\mathcal{H}_\eps(t) =& \frac{1}{2} \int n\vert  u - u_{app} \vert ^2 dx + T^i \int n(\log \left(n/n_{app}\right)  - 1 + n_{app}/n ) dx  \\+&\int (e^{-\phi}\log \left( e^{-\phi}/ n_{app} \right)  - e^{-\phi} +n_{app})dx  
+ \frac{\eps^2}{2} \int \vert \partial_x \phi \vert^2 dx,
\end{split}
\end{equation}
Then, by an explicit computation (similar to the one given in \cite{HK}) of the derivative in time of $\mathcal{H}_\eps$, the principle is to show that this is a Lyapunov function, and thus that it vanishes as $\eps\rightarrow0$, if $\mathcal{H}_\eps(0)\rightarrow_{\eps \rightarrow 0} 0$. This would roughly proves strong convergence in $L^2$ for $(n_\eps, u_\eps, \phi_\eps)$ without performing high-orders estimates in order to obtain strong compactness, as in the proof of this paper. Notice furthermore that this seems anyway impossible to justify such high-order estimates for solutions with low regularity.

The convergence would still be local in time,  since for this method, one has to consider \emph{strong} solutions (that is with at least Lipschitz regularity) to the limit systems.

\subsection{Two Stream instability}

It seems reasonable to consider the case of several species of ions, say two for the clarity of exposure.
This means that each species, described by $(n_1,u_1)$ (resp. $(n_2,u_2)$) satisfies an isothermal Euler-Poisson system:
\begin{equation} \label{EP2}
\left\{
\begin{aligned}
 & \pa_t n_i \: + \:  \div(n_i u_i) \:  = \: 0, \\
 & \pa_t u_i  \: + \:  u_i \cdot \na u_i  \: + \: T^i \,  \na \ln(n_i)= \na \phi, \\
\end{aligned}
\right. 
\end{equation}
and these systems are coupled through the quasineutral Poisson equation:
\begin{equation}
\label{Poisson2}
\eps^2 \Delta \phi \: + \: e^{-\phi}  = n_1 \: + \: n_2. 
\end{equation}

One could expect to prove similar results to those proved in this paper.
Nevertheless, the situation turns out to be very different in that framework (even without boundary). This is due to the so-called two stream instability. 
We refer to the work of Cordier, Grenier and Guo \cite{CGG}, who studied that mechanism and showed that any initial data such that 
$$
u_{1,0} \neq u_{2,0}
$$
is non linearly unstable, hence the name two stream instability (actually they do not explicitly treat the case of \eqref{Poisson2}, but their study also holds in that case). Instability means here  that there exists some $\alpha_0>0$ such that, for any $\delta>0$ , there exists an initial condition in a ball of radius $\delta$ (for a Sobolev norm with regularity as high as we want) and with center the two-stream initial condition and whose associated solution goes out of the ball of radius $\alpha_0$ in the $L^2$ norm, after some time of order $O(|\log \delta|)$. In other words, there are perturbations of the two-stream data which are arbitrarily small in any Sobolev norm (as high as we want) and which are wildly amplified.

One can then observe that it is possible to go from the system \eqref{EP2}-\eqref{Poisson2} with any $\eps>0$ to \eqref{EP2}-\eqref{Poisson2} with $\eps=1$ thanks to the change of variables $(t,x) \mapsto (\frac{t}{\eps}, \frac{x}{\eps})$. This means that the quasineutral limit can somehow be seen a long time behaviour limit for the Euler-Poisson system.
By extrapolating the two-stream instability results, one can infer that the formal limit for \eqref{EP2}-\eqref{Poisson2} is false in Sobolev regularity for such data. 

On the other hand, one can conjecture that the formal limit is true as soon as we avoid two-stream instabilities, that is to say when:
$$
u_{1,0} = u_{2,0}.
$$

%\bibliographystyle{plain}
%\bibliography{GVHKR6}

\end{document}